\documentclass[11pt,reqno]{amsart}

\usepackage[T1]{fontenc}
\usepackage{lmodern}
\usepackage[a4paper,left=25mm,right=25mm,top=25mm,bottom=25mm,
  headheight=14pt,headsep=8mm]{geometry}
\usepackage{amssymb,mathtools}
\usepackage{enumitem}
\usepackage{array,longtable,booktabs,needspace}
\usepackage{microtype}
\usepackage[authoryear,round]{natbib}
\usepackage{xcolor}
\usepackage[colorlinks=true,linkcolor=black,citecolor=blue,
  urlcolor=blue]{hyperref}

\newtheorem{theorem}{Theorem}[section]
\newtheorem{proposition}[theorem]{Proposition}
\newtheorem{lemma}[theorem]{Lemma}
\newtheorem{corollary}[theorem]{Corollary}
\theoremstyle{definition}

\theoremstyle{remark}
\newtheorem{remark}[theorem]{Remark}
\numberwithin{equation}{section}
\AddToHook{cmd/section/before}{\Needspace{5\baselineskip}}
\AddToHook{cmd/subsection/before}{\Needspace{4\baselineskip}}
\AddToHook{env/lemma/before}{\Needspace{6\baselineskip}}
\AddToHook{env/proposition/before}{\Needspace{6\baselineskip}}
\AtBeginDocument{%
  \setlength{\abovedisplayskip}{9pt plus 3pt minus 2pt}%
  \setlength{\belowdisplayskip}{9pt plus 3pt minus 2pt}%
  \setlength{\abovedisplayshortskip}{5pt plus 3pt}%
  \setlength{\belowdisplayshortskip}{6pt plus 3pt minus 2pt}%
}

\newcommand{\E}{\mathbb{E}}
\newcommand{\Z}{\mathbb{Z}}
\newcommand{\N}{\mathbb{N}}
\newcommand{\R}{\mathbb{R}}
\newcommand{\one}{\mathbf{1}}
\DeclareMathOperator{\sgn}{sgn}
\DeclareMathOperator{\dist}{dist}
\newcommand{\symd}{\mathbin{\triangle}}

\title[Extremal cosine sign correlations]{Extremal cosine sign correlations}

\author[Shengyang Zang]{Shengyang Zang}
\address{Shengyang Zang, School of Mathematical Sciences, Nankai University, Tianjin, China}
\email{1120230032@mail.nankai.edu.cn}

\author[Yingyu Zhao]{Yingyu Zhao}
\address{Yingyu Zhao, Qiuzhen College, Tsinghua University, Beijing, China}
\email{zhao-yy25@mails.tsinghua.edu.cn}

\author[Yunhao Zhao]{Yunhao Zhao}
\address{Yunhao Zhao, School of Advanced Interdisciplinary Sciences, University of Chinese Academy of Sciences, Beijing, China}
\email{zhaoyunhao23@mails.ucas.ac.cn}

\author[Chen Zhou]{Chen Zhou}
\address{Chen Zhou, School of Mathematical Sciences, Zhejiang University of Technology, Hangzhou, China}
\email{zcygo01@163.com}

\date{}

\hypersetup{
  pdftitle={Extremal cosine sign correlations},
  pdfauthor={Shengyang Zang, Yingyu Zhao, Yunhao Zhao, Chen Zhou}
}

\begin{document}
\raggedbottom

\begin{abstract}
We prove that four cosines with distinct positive integer frequencies have
the same sign on at least $1/33$ of a period, with equality precisely for
positive integer dilates of $\{1,3,11,33\}$.
If both odd and even frequencies occur, the optimal proportion is $1/24$;
equality holds precisely for odd positive integer dilates of
$\{1,3,4,12\}$ and $\{2,4,5,15\}$.
For arbitrary $n$, we establish the universal lower bound $4^{1-n}$
and construct frequency sets for which this proportion is at most
$16n^\theta4^{-n}$, where $\theta=\log_2(3/2)$.
These bounds determine the optimal exponential decay rate $1/4$.
\end{abstract}

\maketitle

\section{Introduction}

Sign correlations measure the agreement of oscillatory functions after
their amplitudes have been discarded. Even for distinct cosine modes,
whose inner product over a period is zero, the signs can remain
correlated. The positions of their zeros are governed by the
frequencies, and arithmetic relations among those frequencies can
produce substantial departures from an independent-sign model.
Cosine signs therefore provide a concrete setting in which to study
how arithmetic structure affects simultaneous oscillation.

One source of the question is the study of simultaneous sign patterns
of orthogonal functions. In their work on an uncertainty principle
for roots, \citet*{GOS17} proved that, at any prescribed finite set of
real points, infinitely many Hermite polynomials of degree divisible
by four are positive at every point. Their argument uses trigonometric
asymptotics and recurrence of linear flows on a torus. The same work
also exhibits arithmetic restrictions that exclude certain other sign
patterns along this subsequence once the degree is sufficiently large. These
results lead naturally from the occurrence of a sign pattern to the
quantitative question of how often it occurs.

In their subsequent work, \citet*{GOS} established sign correlation
laws for sequences with suitable trigonometric asymptotics, including
eigenfunctions of Schr\"odinger operators and classical orthogonal
polynomials. They fix two spatial points and vary the function index.
Under the stated asymptotic and equidistribution hypotheses, the
limiting proportion of agreement lies between $1/3$ and $2/3$, and
both bounds are sharp. They explicitly asked whether analogous
results hold at three or more points, and which prescribed sign
patterns have densities uniformly separated from $0$ or $1$.
The integer-frequency cosine problem studied by \citet{DGLLP}
isolates the arithmetic trigonometric model behind these questions.
Its spatial average can be analyzed directly, with the frequencies
encoding the dependence among the signs.

A related finite problem arises when the entries of a Fourier matrix
are replaced by their signs: \citet{CS} quantify the resulting failure
of orthogonality for the signed discrete cosine transform. Together,
these developments motivate the study of joint sign patterns beyond
ordinary pairwise orthogonality, and of sharp bounds that remain
valid across different arithmetic configurations.

\Needspace{10\baselineskip}
Let $\N=\{1,2,\ldots\}$. For a nonempty finite list
$A=(a_1,\ldots,a_n)$ of positive integers, define
\[
P(A)=\frac{1}{2\pi}\,\bigl|\{x\in[0,2\pi):
\cos(a_1x),\ldots,\cos(a_nx)\text{ have the same sign}\}\bigr|
\]
and
\[
p_n=\inf_{1\le a_1<\cdots<a_n}P(a_1,\ldots,a_n).
\]
Here $|\cdot|$ denotes Lebesgue measure. We use $P$ for both lists and
sets, since repetitions do not affect the event, and write
$dA=(da_1,\ldots,da_n)$ for $d\in\N$. For each fixed $A$, all the
cosines are positive near the origin, so $P(A)>0$. The interval
guaranteed by this observation shrinks as the largest frequency
increases, and hence does not give a positive lower bound depending
only on $n$. The extremal problem asks both how small $P(A)$ can be
for a fixed number of frequencies and how its infimum decays as that
number grows.

Each individual sign is equally likely to be positive or negative.
For $n$ independent fair signs, the probability of complete agreement
would be $2^{1-n}$. In the present problem all signs share the same
parameter, and their dependence can make agreement much less
frequent. Thus even the exponential scale of $p_n$ requires an
understanding of the joint distribution. For a fixed $n$, one can
also ask whether the infimum is attained and which arithmetic
configurations realize it.

There is also a geometric interpretation. With $t=x/(2\pi)$, regard
$a_1t,\ldots,a_nt$ modulo $1$ as the positions of runners on a circular
track. Then $P(A)$ is the proportion of time for which all runners
lie in the fixed semicircle $(-1/4,1/4)$ or all lie in its complement.
This interpretation was noted by \citet{DGLLP}. Its broader geometric
setting has a longer history in the Diophantine approximation and
view-obstruction problems of \citet{Wills} and \citet{Cusick}, which
gave rise to the lonely-runner conjecture.

Several subsequent developments in that literature concern themes
relevant here. \citet{GW} investigated tight sets of runner speeds;
\citet{PerarnauSerra} used correlations between the times at which
runners lie close to the origin to improve bounds for simultaneous
separation; and \citet{Kravitz} studied possible values of maximum
loneliness and proposed a discrete spectrum near the conjectured
minimum. These works illustrate the roles of arithmetic dependence
and extremal configurations in restrictions on a torus orbit.
Maximum loneliness optimizes a minimum circular distance over time.
The quantity $P(A)$ measures the time spent in a fixed union of two
cubes. Its extremal problem therefore requires uniform control of
this time average as the integer direction of the orbit varies.

The two-frequency calculation of \citet*{GOS} gives $p_2=1/3$,
with equality precisely for positive integer dilates of $\{1,3\}$.
The next exact case was settled by \citet*{DGLLP}, whose work first
circulated in 2022 and was published in 2024. They proved $p_3=1/9$
and identified every minimizer as a positive integer dilate of
$\{1,3,9\}$. Thus the three-frequency result determined both the
sharp probability and the arithmetic structure of all equality cases.

Their proof also provided a way to reduce an unbounded frequency
search to a finite one. After removing the common divisor, Fourier
formulas for pair correlations and arithmetic reductions force any
triple $a<b<c$ with $P(a,b,c)\le1/9$ to have all three entries odd
and $a=1$. A scale-separation estimate shows that a third frequency
much larger than the first two contributes asymptotically a factor
$1/2$ to the agreement probability. Combining this estimate with the
pair correlation bounds yields $b\le7$ and $c\le84$, after which a
finite computation completes the proof. The key issue is thus the
relative scales and reduced ratios of the frequencies: a common
dilation leaves the probability unchanged.

As noted by \citet{DGLLP}, the restriction of extremizers to
controlled scales has an analogue in the compactness principle of
\citet{SteinerbergerCompactness}. For smooth mean-zero functions on
the two-dimensional torus, that principle bounds the length of a
closed geodesic maximizing the absolute average of the function.
In the sign problem, the discontinuous indicators call for direct
correlation and interval estimates. The three-frequency argument
shows how these estimates can supply the quantitative restrictions
needed for an exact extremal classification.

The same work of \citet{DGLLP} also identified the change in behavior
at four frequencies. Numerical exploration followed by exact
evaluation of candidate configurations gave
\[
P(1,3,11,33)=\frac1{33}<\frac1{27}=P(1,3,9,27)
\]
and led to the conjecture that the positive integer dilates of
$\{1,3,11,33\}$ are the only minimizers for four frequencies.
Thus the geometric
progression with ratio $3$, which is optimal for two and three
frequencies, already ceases to be optimal at the next step. Four
frequencies are also the first case in which the common-sign
probability involves a fourth-order correlation in addition to the
pair correlations. Our first result resolves the conjecture and
classifies the extremal configurations in this first higher-order
case.

\begin{theorem}\label{thm:main}
Let $a_1,a_2,a_3,a_4$ be distinct positive integers. Then
\[
P(a_1,a_2,a_3,a_4)\ge\frac1{33}.
\]
Equality holds if and only if
\[
\{a_1,a_2,a_3,a_4\}=d\{1,3,11,33\}
\quad\text{for some }d\in\N.
\]
In particular, $p_4=1/33$.
\end{theorem}

Parity distinguishes two different common-sign geometries. For odd
frequencies, translation by half a period reverses every sign and
exchanges the common positive and common negative sets. When both
parities occur, the same translation reverses only the odd signs,
so agreement within the two parity classes must also be coordinated.
The sharp constant under this restriction records the effect of
that change in symmetry.

% Keep the introductory sentence with the theorem.
\Needspace{9\baselineskip}
We also determine the minimum when both parities occur.

\begin{theorem}\label{thm:mixed}
Let $a_1,a_2,a_3,a_4$ be distinct positive integers, at least one odd
and at least one even. Then
\[
P(a_1,a_2,a_3,a_4)\ge\frac1{24}.
\]
Equality holds if and only if
\[
\{a_1,a_2,a_3,a_4\}=d\{1,3,4,12\}
\quad\text{or}\quad d\{2,4,5,15\}
\]
for an odd positive integer $d$.
\end{theorem}

The failure of the powers of $3$ to minimize $P$ at four frequencies
also suggests that their decay rate need not describe the problem
for large $n$. The five-frequency candidate
$\{1,3,11,35,105\}$ reported by \citet{DGLLP} gives the further upper
bound $p_5\le1/105$. For arbitrary $n$, they proved
$p_n\le P(1,3,\ldots,3^{n-1})=3^{1-n}$,
left open whether $p_n>0$ for every $n$, and suggested decay faster than
$3^{-n}$. Becker obtained $p_n\ge4^{-n}$ by a positive definite
minorant argument, as recorded in the IMPA lectures of
\citet[pp.~120--124]{Steinerberger}. This leaves the question of
whether the scale $4^{-n}$ can be approached by actual frequency
sets. We determine the exponential rate and bound the remaining
factor by a power of $n$.

\begin{theorem}\label{thm:rate}
Let $\theta=\log_2(3/2)$. For every $n\in\N$,
\[
4^{1-n}\le p_n\le16n^\theta4^{-n}.
\]
In particular,
\[
\lim_{n\to\infty}p_n^{1/n}=\frac14,
\qquad
\log p_n=-n\log4+O(\log n).
\]
\end{theorem}

The theorem shows that arithmetic dependence can make simultaneous
agreement exponentially rarer than the independent-sign benchmark
$2^{1-n}$. It also separates two aspects of the extremal problem:
the four-frequency theorems determine sharp constants and every
equality configuration, while Theorem~\ref{thm:rate} determines the
exponential scale for arbitrarily many frequencies. The remaining
gap is polynomial; these estimates do not determine the precise
factor multiplying $4^{-n}$.

The bounds for general $n$ follow from a difference-set argument and
an iterative construction. The four-frequency results also require
arithmetic estimates for pair correlations and conditional averages.

\subsection{The exponential rate}

The positive definite minorant approach connects the uniform
positivity question with classical extremal problems in harmonic
analysis. A continuous positive definite function on a torus has
nonnegative Fourier coefficients, so its average along an
integer-frequency orbit through the identity is at least its Haar
average. A related support-constrained problem is Tur\'an's extremal
problem, which maximizes the integral of a positive definite function
supported in a prescribed region and normalized at the identity; see
\citet{KR}. In our lower bound, the needed uniformity comes from
the difference-set structure of the common-sign region and the
classical sumset inequality on compact connected abelian groups. The torus
case was established by \citet{Macbeath}, and the general framework
is due to \citet{Kneser}; see also \citet{Tao}.

For the lower bound, view the common-sign event as the inverse image of
two cubes under the homomorphism
\[
t\longmapsto(a_1t,\ldots,a_nt)\in(\R/\Z)^n.
\]
Their union is $D-D$, where $D$ is a union of two smaller cubes of total
measure $2\cdot4^{-n}$. By averaging, some translated section of $D$ has
at least that measure. Taking differences returns to a section through
the origin, and the circle sumset inequality doubles its measure.
This gives $P(A)\ge4^{1-n}$ without additional arithmetic assumptions
on the frequencies. The argument uses the size of a translated
section even when the original section of $D$ is small, and so avoids
any dependence on the largest frequency.

For the upper bound, we use the cyclic construction of Polak and
Schrijver recorded in \citet[Theorem~9.3.3]{Polak}. Its separation
property allows us to add a block of $m$ odd frequencies and retain
just one of the arcs produced by dilation. The probability is divided by
\[
Q_m=\frac{2\cdot4^m+1}{3},
\]
at a cost of at most $3/2$ relative to the factor $4^{-m}$.
The phase points $j(1,4,\ldots,4^{m-1})/Q_m$, for
$j\in\Z/Q_m\Z$, are pairwise separated by more than $1/4$ in
circular distance in at least one coordinate. Thus the axis-parallel
cubes of side length $1/4$
centered at these points are disjoint. Their number $Q_m$ is within
a constant factor of the volume bound $4^m$, which explains the
scale relevant to our construction.

Such cyclic separation is also used to construct independent sets
in strong powers of circular graphs in the study of Shannon
capacity \citep{PS}. The lattice constructions of \citet*{BPZ}
include the theorem of Polak and Schrijver used here. For our
problem, the additional step is to preserve the discrete separation
when grid points are replaced by intervals of a suitable width.
We give a base-four proof of the separation property, pass from the
cyclic grid to intervals of the circle, and iterate with increasing
block lengths. The number of iterations is at most $\log_2 n+O(1)$,
which accounts for the exponent $\theta=\log_2(3/2)$.

\Needspace{10\baselineskip}
\subsection{The four-frequency classification}

Write
\[
s(x)=\sgn\cos x,\qquad
\E f=\frac1{2\pi}\int_0^{2\pi}f(x)\,dx,
\]
and let $S(A)$ be the sum of the pair correlations and $Q(A)$ the
expectation of the product of all signs. The indicator expansion gives
\[
4P(a,b,c)=1+S(a,b,c),\qquad
8P(a,b,c,d)=1+S(a,b,c,d)+Q(a,b,c,d).
\]
Thus the passage from three to four frequencies introduces a
fourth-order term that is not determined by the pair correlations.
For odd frequencies, a multiplicative identity under translation and
doubling yields $Q(A)\ge-1/3$. Consequently,
\[
P(A)\le\frac1{33}
\quad\Longrightarrow\quad S(A)\le-\frac{14}{33}.
\]
The pair correlation formula expresses each negative summand as
$-1/(uv)$, where $u:v$ is the reduced frequency ratio. Under the
hypothesis $P(A)\le1/33$, this constraint, together with one exceptional
interval calculation, forces a pair of the form $\{a,3a\}$, which we
call a \emph{$3$-pair}.

Two disjoint $3$-pairs can be evaluated explicitly and give the
extremizer $\{1,3,11,33\}$. For a single $3$-pair, the correlation
inequality bounds the reduced ratios and their common scale. When one
frequency remains large, we use the component estimate
\[
P(T,c)\ge\frac12P(T)-\frac{N(T)}{4c}
\]
where $T$ is a triple and $N(T)$ counts its common-sign arcs.
Together, these estimates reduce the
remaining configurations to a finite list, which is excluded by
explicit common positive intervals.

In mixed parity, translation by $\pi$ reverses the odd signs and
preserves the even signs. If one parity occurs only once, the
three-frequency theorem gives $P\ge1/18$. We are therefore left with
two odd and two even frequencies. In the presence of an odd $3$-pair,
conditioning on its common-sign set gives
\[
P=\frac{1+\rho}{12},
\]
where $\rho$ is a two-sign correlation averaged over a cyclic grid and
a central interval. The sharp estimate $\rho\ge-1/2$, including its
equality cases, produces the two extremizers in
Theorem~\ref{thm:mixed}. The other configurations are treated using
correlations on cyclic groups and the short-ratio estimates.

\subsection{Organization}

Section~\ref{sec:many} proves Theorem~\ref{thm:rate} independently of
the four-frequency analysis. Section~\ref{sec:basic} develops the
correlation identities and reductions used for four frequencies.
Section~\ref{sec:bernoulli} gives the short-ratio formulas used in the
odd-frequency classification in Section~\ref{sec:odd}.
Section~\ref{sec:auxiliary} establishes the conditional correlation
estimate and its supporting interval and grid bounds;
Section~\ref{sec:mixed} applies them to mixed parity.
Section~\ref{sec:final} completes the proofs of the two
four-frequency theorems. The interval calculations and the arithmetic
reduction for a single $3$-pair are recorded in
Appendices~\ref{sec:odd-certificates} and~\ref{sec:odd-arithmetic}.

\section{The exponential rate}\label{sec:many}

We use both the angular variable $x\in\R/(2\pi\Z)$
and the normalized variable $t=x/(2\pi)\in\R/\Z$; $|\cdot|$ denotes
Lebesgue measure in the indicated variable, and cardinality for finite
sets. All identities between sign functions and sets are understood
up to null sets.

Integer dilation preserves normalized Lebesgue measure on the circle, so
\begin{equation}\label{eq:dilation}
P(dA)=P(A), \quad d\in\N.
\end{equation}
For a nonempty finite set $A$, write
\[
A_+=\{t\in\R/\Z:\cos(2\pi at)>0\text{ for all }a\in A\},
\]
and define $A_-$ by reversing the inequality. If all frequencies are odd,
translation by $1/2$ exchanges these sets, while reflection preserves
each of them. Hence
\begin{equation}\label{eq:odd-symmetry}
A_-=A_++\frac12,\qquad
P(A)=2|A_+|=4|A_+\cap[0,1/2]|.
\end{equation}
The interval $(-1/(4\max A),1/(4\max A))$ lies in $A_+$; in particular,
\begin{equation}\label{eq:initial-arc}
P(A)\ge\frac1{\max A},\quad A\subset2\N-1.
\end{equation}

\subsection{The lower bound}

All tori in this section carry normalized Haar measure.

There is a translated section of $D$ with measure at least $|D|$. Taking its
difference set removes the translation, so the circle sumset inequality
gives a lower bound for the section of $D-D$ through the origin.

\begin{lemma}\label{lem:slice-difference}
Let $D,V$ be open subsets of $(\R/\Z)^n$ with $D-D\subset V$, and let
$a\in\Z^n$. Then
\[
\bigl|\{t\in\R/\Z:at\in V\}\bigr|
\ge\min\{2|D|,1\}.
\]
\end{lemma}

\begin{proof}
For $z\in(\R/\Z)^n$, let
$E_z=\{t\in\R/\Z:z+at\in D\}$. By Fubini's theorem,
\[
\int_{(\R/\Z)^n}|E_z|\,dz=|D|.
\]
Choose $z$ with $|E_z|\ge|D|$. Since
\[
E_z-E_z\subset\{t:at\in D-D\}\subset\{t:at\in V\},
\]
the circle case of Kneser's sumset inequality
\citep[eq.~(1)]{Tao} gives
\[
\bigl|\{t:at\in V\}\bigr|
\ge |E_z-E_z|
\ge \min\{2|E_z|,1\}
\ge \min\{2|D|,1\}.
\]
Here the open-set version follows by inner regularity.
\end{proof}

\begin{proof}[Proof of the lower bound in Theorem~\ref{thm:rate}]
Let $A=\{a_1,\ldots,a_n\}$, $a=(a_1,\ldots,a_n)$, and
$\mathbf1=(1,\ldots,1)$. In $(\R/\Z)^n$, let
\[
C=(-1/8,1/8)^n,\qquad
D=C\cup\bigl(C+\tfrac12\mathbf1\bigr).
\]
Then $|D|=2\cdot4^{-n}$ and
\[
D-D=(-1/4,1/4)^n
\cup\bigl(\tfrac12\mathbf1+(-1/4,1/4)^n\bigr).
\]
The two sets on the right are the common positive and common negative
regions of the coordinate cosines. Hence Lemma~\ref{lem:slice-difference}
gives
\[
P(A)\ge 2|D|=4^{1-n}.
\]
Taking the infimum over $A$ proves the claim.
\end{proof}

\subsection{A cyclic construction}

For a finite nonempty set $A$ of odd positive integers,
\eqref{eq:odd-symmetry} gives $P(A)=2|A_+|$.
We call $A$ an \emph{arc set of parameter $X$} if $X\ge4$ and
\[
A_+=(-1/X,1/X)+\Z
\quad\text{up to a null set}.
\]
Necessarily $X=4\max A$, since the first positive zero among these
cosines is $1/(4\max A)$. In particular,
\begin{equation}\label{eq:arcparameter}
P(A)=\frac4X.
\end{equation}

For $m\ge0$, let
\[
Q_m=\frac{2\cdot4^m+1}{3}.
\]
Thus $Q_0=1$ and $Q_m=4Q_{m-1}-1$ for $m\ge1$.
The following lemma is the $r=4$ case of the cyclic construction of
Polak and Schrijver recorded in \citet[Theorem~9.3.3]{Polak}.
We give a proof in base four.

\begin{lemma}\label{lem:cyclic-isolation}
For every $m\ge1$ and $j\not\equiv0\pmod{Q_m}$,
\[
\max_{0\le i<m}\dist(4^ij,Q_m\Z)
\ge Q_{m-1}>\frac{Q_m}{4}.
\]
\end{lemma}

\begin{proof}
Let $q=Q_m$ and $1\le j<q$. Suppose that all $m$ distances are less
than $q/4$. Then the first $m$ base-four digits of $j/q$ belong to
$\{0,3\}$, so $K=\lfloor4^mj/q\rfloor$ is divisible by $3$.
On the other hand, $4^m=(3q-1)/2$ gives
\[
K=\left\lfloor\frac{3j}{2}-\frac{j}{2q}\right\rfloor
=\begin{cases}
3h-1,&j=2h,\\[3pt]
3h+1,&j=2h+1.
\end{cases}
\]
Both cases contradict $3\mid K$. Since $q\equiv3\pmod4$, an integer
distance at least $q/4$ is at least $(q+1)/4=Q_{m-1}$.
\end{proof}

For $m\ge1$, let
\[
D_m=\{1\}\cup\{Q_m-4^i:1\le i<m\},\qquad d_m=\max D_m.
\]
For $1\le i<m$, replacing $4^i$ by $Q_m-4^i$ preserves the cosine values on the
$Q_m$-point grid and makes all frequencies odd. Thus $D_m$ consists of
$m$ distinct odd positive integers smaller than $Q_m$, and
\[
d_1=1,\quad d_m=Q_m-4,\quad m\ge2.
\]

\begin{proposition}\label{prop:prefix-family}
Let $m\ge1$, and let $B$ be an arc set of parameter $X\ge4d_m$.
For $A=D_m\cup Q_mB$, we have
\[
|A|=|B|+m,\qquad
A_+=\left(-\frac1{Q_mX},\frac1{Q_mX}\right)+\Z
\quad\text{up to a null set}.
\]
In particular, $A$ is an arc set of parameter $Q_mX$.
\end{proposition}

\begin{proof}
Let $q=Q_m$ and $\beta=1/X$. Then, up to a null set,
\begin{equation}\label{eq:sampling}
(qB)_+=\bigcup_{j\in\Z}
\left(\frac{j-\beta}{q},\frac{j+\beta}{q}\right)\pmod1.
\end{equation}
At $j/q$, the cosines with frequencies in $D_m$ have the same values
as those with frequencies $1,4,\ldots,4^{m-1}$.
By Lemma~\ref{lem:cyclic-isolation}, they are all positive only when
$j\equiv0\pmod q$.

For $a\in\N$ and $k\in\Z$, the oddness of $q$ gives
\[
\left|\frac jq-\frac{2k+1}{4a}\right|
=\frac{|4aj-q(2k+1)|}{4aq}\ge\frac1{4aq}.
\]
Since $\beta\le1/(4d_m)$, no cosine with frequency in $D_m$ changes
sign inside an interval of \eqref{eq:sampling}. Thus only the central
interval remains:
\[
(D_m\cup qB)_+=(-\beta/q,\beta/q)+\Z
\quad\text{up to a null set}.
\]
The blocks are disjoint since $\max D_m<q$, giving $|B|+m$ elements.
\end{proof}

\subsection{Iteration with increasing block lengths}

Proposition~\ref{prop:prefix-family} adds $m$ frequencies and divides
the common-sign probability by $Q_m$. Since
\[
Q_m\ge\frac23\,4^m,\qquad
P(D_m\cup Q_mB)=\frac{P(B)}{Q_m}
\le\frac32\,4^{-m}P(B),
\]
each step costs at most a factor $3/2$ relative to the scale $4^{-m}$.
We choose the largest admissible block at every step. The resulting
cardinalities satisfy $n_j\ge\theta2^j$, so the number of steps is at
most $\log_2 n+O(1)$. This yields the factor $n^{\log_2(3/2)}$.
For $n_j<n<n_{j+1}$, we use a smaller final block.

\begin{proof}[Proof of the upper bound and limit in Theorem~\ref{thm:rate}]
Put $\delta=\theta/2=\log_4(3/2)$.
The identities
\[
\{1,3\}_+=(-1/12,1/12)+\Z,\qquad
\{1,3,9\}_+=(-1/36,1/36)+\Z
\]
follow by intersecting the positive intervals of the indicated cosines.
Let $B_0=\{1,3,9\}$, $n_0=3$, and $X_0=36$.
Given $B_j$, choose the largest integer $m_j\ge1$ with
$4d_{m_j}\le X_j$, and let
\[
\begin{gathered}
B_{j+1}=D_{m_j}\cup Q_{m_j}B_j,\qquad
n_{j+1}=n_j+m_j,\qquad X_{j+1}=Q_{m_j}X_j.
\end{gathered}
\]
By Proposition~\ref{prop:prefix-family}, $B_j$ is an arc set of
cardinality $n_j$ and parameter $X_j$.

Let $L_j=\log_4X_j$. Since $Q_m\ge(2/3)4^m$,
\begin{equation}\label{eq:parameter-growth}
L_j\ge\log_436+n_j-3-\delta j.
\end{equation}
Since $X_j\ge36$, we have $m_j\ge2$. For $m\ge2$, the condition
$4d_m\le X_j$ is equivalent to
\[
4^m\le\frac{3X_j+44}{8}.
\]
Hence $m_j\ge L_j-\log_4(8/3)-1$.
Since $\log_436=2+\theta$ and $\log_4(8/3)=1-\delta$,
\eqref{eq:parameter-growth} gives
\[
n_{j+1}\ge2n_j-3+3\delta-\delta j.
\]
With $b_j=n_j-\delta j-3+\theta$, this recurrence becomes
$b_{j+1}\ge2b_j$, and $b_0=\theta$. Hence
\begin{equation}\label{eq:cardinality-growth}
n_j\ge\theta2^j+\delta j+3-\theta\ge\theta2^j.
\end{equation}

Given $n\ge3$, choose $j$ with $n_j\le n<n_{j+1}$.
If $n=n_j$, use $B_j$. Otherwise, $m=n-n_j<m_j$ gives
$4d_m\le4d_{m_j}\le X_j$, so Proposition~\ref{prop:prefix-family}
applies. In either case we obtain an arc set
of $n$ elements with parameter
\[
X\ge36\cdot4^{n-3}\left(\frac{2}{3}\right)^{j+1}.
\]
By \eqref{eq:arcparameter} and \eqref{eq:cardinality-growth},
\[
p_n\le\frac4X
\le\frac{32}{3}\left(\frac{3}{2}\right)^j4^{-n}
\le\frac{32}{3}\theta^{-\theta}n^\theta4^{-n}.
\]
Since $1/2<\theta<1$ and $t^t$ increases on $[1/2,1]$,
$\theta^{-\theta}\le\sqrt2$. The inequality $32\sqrt2/3<16$ proves
the upper bound for $n\ge3$; for $n=1,2$ it follows from $p_n\le1$.

Taking $n$th roots gives
\[
4^{1/n-1}\le p_n^{1/n}
\le\frac14\,16^{1/n}n^{\theta/n}.
\]
The outer terms tend to $1/4$, proving the limit. Taking logarithms
also gives $\log p_n=-n\log4+O(\log n)$.
\end{proof}

\section{Reduction of the four-frequency problem}\label{sec:basic}

We express $P$ in terms of pair correlations and one fourth-order
correlation. We first record their algebraic relations,
then remove configurations with mixed parity of type $1+3$ and those
with a common factor among three frequencies. The last estimate in this
section controls the fourth-order term for odd frequencies.

\subsection{Correlation identities}

For positive integers $a,b$ and a nonempty finite list
$A=(a_1,\ldots,a_n)$, let
\[
C(a,b)=\E[s(ax)s(bx)],\qquad
Q(A)=\E\prod_{j=1}^n s(a_jx),\qquad
S(A)=\sum_{i<j}C(a_i,a_j).
\]
For odd $m$, let $\chi(m)=\chi_4(m)=(-1)^{(m-1)/2}$;
for $m\equiv\pm1\pmod 6$, let $\chi_6(m)=\pm1$,
where $\chi_4$ and $\chi_6$ are the nonprincipal
Dirichlet characters modulo $4$ and $6$, respectively.
Our Fourier convention is
\[
\widehat f(k)=\frac1{2\pi}\int_0^{2\pi}f(y)e^{-iky}\,dy,
\qquad k\in\Z.
\]

\begin{lemma}\label{lem:expansion}
For positive integers $a_1,\ldots,a_4$,
\[
\begin{aligned}
2P(a_1,a_2)&=1+C(a_1,a_2),\\
4P(a_1,a_2,a_3)&=1+S(a_1,a_2,a_3),\\
8P(a_1,a_2,a_3,a_4)&=1+S(a_1,a_2,a_3,a_4)+Q(a_1,a_2,a_3,a_4).
\end{aligned}
\]
\end{lemma}

\begin{proof}
For $\epsilon_1,\ldots,\epsilon_n\in\{\pm1\}$,
\[
\one[\epsilon_1=\cdots=\epsilon_n]
=2^{-n}\left(\prod_{j=1}^n(1+\epsilon_j)
             +\prod_{j=1}^n(1-\epsilon_j)\right).
\]
The odd-degree terms cancel. Substitute $\epsilon_j=s(a_jx)$ and integrate.
\end{proof}

\begin{lemma}\label{lem:pair}
Let $a,b\in\N$, and write $a=gu$, $b=gv$, where $\gcd(u,v)=1$.
Then
\[
C(a,b)=
\begin{cases}
0,&uv\text{ even},\\[4pt]
\dfrac{\chi(u)\chi(v)}{uv},&uv\text{ odd}.
\end{cases}
\]
In particular, for odd
$a,b$,
\begin{equation}\label{eq:pairodd}
C(a,b)=\chi(a)\chi(b)\,\frac{\gcd(a,b)^2}{ab}.
\end{equation}
\end{lemma}

\begin{proof}
Integer dilation gives $C(a,b)=C(u,v)$. The Fourier coefficients of $s$ are
\[
\widehat s(k)=
\begin{cases}
2\chi(k)/(\pi k),&k\text{ odd},\\[3pt]
0,&k\text{ even}.
\end{cases}
\]
The matching indices in Parseval's identity satisfy
$um=vm'$, hence $(m,m')=(vk,uk)$. If $uv$ is even, every product of
coefficients vanishes. Otherwise,
\[
C(a,b)=\frac{4\chi(u)\chi(v)}{\pi^2uv}
       \sum_{\substack{k\in\Z\\k\text{ odd}}}\frac1{k^2}
       =\frac{\chi(u)\chi(v)}{uv}.
\]
Here $\sum\limits_{k\text{ odd}}k^{-2}=\pi^2/4$, with $k\in\Z$.
For odd $a,b$, the identity
$\chi(a)\chi(b)=\chi(g)^2\chi(u)\chi(v)$ proves \eqref{eq:pairodd}.
\end{proof}

For odd $a,b$, negative correlations occur precisely when
$a\not\equiv b\pmod4$. They equal $-1/(uv)$, where $u:v$ is the
reduced ratio and $uv\equiv3\pmod4$. The three smallest products,
$3,7,11$, correspond to ratios $1:3$, $1:7$, $1:11$.

\subsection{Parity and primitive quadruples}

Suppose $A$ contains both odd and even frequencies, and let $O$ and $E$
be its odd and even subsets.
For any nonempty set $T$, put
\[
G_T=\{x\in\R/(2\pi\Z):s(tx)\text{ is independent of }t\in T\}.
\]
We write $\Pr$ for normalized Lebesgue measure on this circle.
The translation $x\mapsto x+\pi$ reverses the common sign on $O$
and preserves it on $E$. This gives the following reduction.

\begin{lemma}\label{lem:dilation}
If $O,E\ne\varnothing$, then
\begin{equation}\label{eq:parity-split}
P(A)=\frac12\Pr(G_O\cap G_E).
\end{equation}
\end{lemma}

\begin{proof}
Partition $G_O\cap G_E$ into $G_+$ and $G_-$ according as the two
common signs agree or differ. Since
\[
s(a(x+\pi))=(-1)^a s(ax),\qquad G_-=G_++\pi,
\]
we have
\[
P(A)=\Pr(G_+)=\Pr(G_-)=\frac12\Pr(G_O\cap G_E).\qedhere
\]
\end{proof}

A quadruple is \emph{primitive} if its frequencies have greatest common
divisor $1$. By dilation, it suffices to prove Theorem~\ref{thm:main} for
primitive quadruples. We recall the three-frequency theorem of
\citet{DGLLP} and give a short proof in the present notation.

\begin{proposition}\label{prop:three}
For distinct $a,b,c\in\N$,
\[
P(a,b,c)\ge\frac19,
\]
with equality if and only if $\{a,b,c\}=d\{1,3,9\}$ for some $d\in\N$.
\end{proposition}

\begin{proof}
Divide by the greatest common divisor. If the parities are mixed, let
$u,v$ be the two frequencies of the same parity.
Equation~\eqref{eq:parity-split} and the pair correlation formula give
\[
P(a,b,c)=\frac12P(u,v)=\frac14(1+C(u,v))\ge\frac16.
\]
Otherwise all three frequencies are odd. There are at most two negative
correlations, since they join different residue classes modulo $4$.
Unless both arise from $3$-pairs,
\[
S\ge-\frac13-\frac17>-\frac59.
\]
Two $3$-pairs share a frequency $y$, whose only possible neighbours are
$y/3$ and $3y$. Thus the triple is $d\{1,3,9\}$, with
\[
S=-\frac13-\frac13+\frac19=-\frac59.
\]
The conclusion follows from $4P=1+S$.
\end{proof}

\begin{corollary}\label{cor:oneodd}
Let $A$ consist of four distinct positive integers. If exactly one or
exactly three frequencies in $A$ are odd, then
\[
P(A)\ge\frac1{18}.
\]
\end{corollary}

\begin{proof}
Let $T$ be the three frequencies of the same parity. By
Equation~\eqref{eq:parity-split} and Proposition~\ref{prop:three},
$P=\frac12P(T)\ge1/18$.
\end{proof}

A primitive quadruple has an odd frequency. The remaining parity
patterns are therefore four odd frequencies, or two odd and two even
frequencies.

\subsection{Common factors and components}

To treat common factors, we first compute the average of $s$ over
an equally spaced grid.

\begin{lemma}\label{lem:average}
For every $g\in\N$ and almost every $z\in\R$,
\[
\frac1g\sum_{j=0}^{g-1}s\Bigl(z+\frac{2\pi j}{g}\Bigr)
=\begin{cases}
0,&g\text{ even},\\[6pt]
\dfrac{\chi(g)}g s(gz),&g\text{ odd}.
\end{cases}
\]
\end{lemma}

\begin{proof}
The average retains precisely the Fourier coefficients with index
divisible by $g$. If $g$ is even they all vanish; if $g$ is odd, then
$\widehat s(gn)=\chi(g)\widehat s(n)/g$.
\end{proof}

For a nonempty finite list $T$ of positive integers, let
\[
h_T(x)=
\begin{cases}
1,&s(tx)=1\text{ for all }t\in T,\\[3pt]
-1,&s(tx)=-1\text{ for all }t\in T,\\[3pt]
0,&\text{otherwise}.
\end{cases}
\]
Then $|h_T|$ is the common-sign indicator, and for $c\in\N$,
\begin{equation}\label{eq:split}
P(T,c)=\frac12P(T)+\frac12\E[h_T(x)s(cx)],
\end{equation}
since the corresponding indicator is $(|h_T|+h_Ts(c\,\cdot))/2$.

The averaging identity gives the following formula for $P(gT,c)$.

\begin{proposition}\label{prop:transfer}
Let $T$ be a nonempty finite list of positive integers, and let $g,c\in\N$
satisfy $g\ge2$ and $\gcd(g,c)=1$. Then
\begin{equation}\label{eq:transfer}
P(gT,c)=\begin{cases}
\dfrac12P(T),&g\text{ even},\\[8pt]
\dfrac12P(T)+\dfrac{\chi(g)}g\left(P(T,c)-\dfrac12P(T)\right),&g\text{ odd}.
\end{cases}
\end{equation}
\end{proposition}

\begin{proof}
By \eqref{eq:split} and dilation,
\[
P(gT,c)=\frac12P(T)+\frac12\E[h_T(gx)s(cx)].
\]
Condition on $y=gx\bmod2\pi$. The preimages are
$x=(y+2\pi j)/g$, $0\le j<g$, and multiplication by $c$ permutes the
residues modulo $g$. Lemma~\ref{lem:average} therefore gives
\[
\begin{aligned}
\E[h_T(gx)s(cx)]
&=\E\left[h_T(y)\frac1g\sum_{j=0}^{g-1}
 s\Bigl(\frac{cy}{g}+\frac{2\pi j}{g}\Bigr)\right]
=\begin{cases}
0,&g\text{ even},\\[6pt]
\dfrac{\chi(g)}g\E[h_T(y)s(cy)],&g\text{ odd}.
\end{cases}
\end{aligned}
\]
Together with \eqref{eq:split}, this proves \eqref{eq:transfer}.
\end{proof}

Since $0\le P(T,c)\le P(T)$, the transfer identity implies
\[
P(gT,c)\ge
\begin{cases}
\dfrac12P(T),&g\text{ even},\\[5pt]
\dfrac{1-1/g}{2}P(T),&g\text{ odd}.
\end{cases}
\]
If three frequencies of a primitive quadruple have a common factor
$g>1$, Proposition~\ref{prop:three} therefore gives
\begin{equation}\label{eq:common-factor-bound}
P\ge
\begin{cases}
\dfrac1{18},&g\text{ even},\\[5pt]
\dfrac{1-1/g}{18}\ge\dfrac1{27},&g\text{ odd}.
\end{cases}
\end{equation}
Both bounds exceed $1/33$.
For Theorem~\ref{thm:mixed}, Corollary~\ref{cor:oneodd} handles one or
three odd frequencies. With two odd and two even frequencies, a common
factor $g>1$ of three frequencies is odd, and the reduced triple still
has mixed parity. Its probability is at least $1/6$, so
\begin{equation}\label{eq:mixed-transfer}
P\ge\frac{1-1/g}{2}\cdot\frac16\ge\frac1{18}>\frac1{24}.
\end{equation}
For the remaining four-frequency arguments, we may therefore assume
\begin{equation}\label{eq:tripgcd}
\gcd(a_i,a_j,a_k)=1\qquad\text{whenever }i,j,k\text{ are distinct}.
\end{equation}

We shall also need a bound when a single frequency is large. In this
case cancellation on each common-sign interval gives an error controlled
by the number of intervals.

Up to finitely many points, the support of $h_T$ is a disjoint union of
open arcs on which $h_T$ is constant. Let $N(T)$ be the least number of
such arcs. Their right endpoints are distinct zeros of the cosines, so
\[
N(T)\le2\sum_{t\in T}t.
\]

\begin{lemma}\label{lem:component}
For every nonempty finite list $T$ of positive integers and every
$c\in\N$,
\begin{equation}\label{eq:component}
P(T,c)\ge\frac12P(T)-\frac{N(T)}{4c}.
\end{equation}
\end{lemma}

\begin{proof}
A primitive of $s(cx)$ is a triangle wave of total range $\pi/c$. Thus,
for every arc $J$,
\[
\left|\int_Js(cx)\,dx\right|\le\frac\pi c.
\]
Summing over the arcs gives
\[
\bigl|\E[h_Ts(c\,\cdot)]\bigr|\le\frac{N(T)}{2c},
\]
and \eqref{eq:split} proves \eqref{eq:component}.
\end{proof}

\subsection{The fourth-order term at odd frequencies}\label{sec:Q}

The pair correlations alone do not determine $P$ for four frequencies.
The remaining term admits a uniform lower bound, independent of the
frequencies. Its proof uses a multiplicative identity under translation
and doubling; the same argument applies to any even number of odd
frequencies.

\begin{proposition}\label{thm:Q}
For every $k\in\N$ and all odd positive integers $a_1,\ldots,a_{2k}$,
\[
\frac1{2\pi}\int_0^{2\pi}\prod_{j=1}^{2k}\sgn\cos(a_jx)\,dx
\ge-\frac13.
\]
\end{proposition}

\begin{proof}
For odd $a$ and almost every $x$,
\begin{equation}\label{eq:doubling}
s(ax)s\bigl(a(x+\pi/2)\bigr)=s\bigl(a(2x+\pi/2)\bigr).
\end{equation}
Both sides equal $-\chi(a)\sgn\sin(2ax)$.
Let $F(x)=\prod_{j=1}^{2k}s(a_jx)$ and $Q=\E F$. Then
\[
F(x)F(x+\pi/2)=F(2x+\pi/2),
\]
so invariance under translation and doubling yields
\[
\E F(x)=\E F(x+\pi/2)=\E[F(x)F(x+\pi/2)]=Q.
\]
Therefore
\[
0\le\E[(1+F(x))(1+F(x+\pi/2))]=1+3Q.\qedhere
\]
\end{proof}

\begin{corollary}\label{cor:S}
For four odd positive integer frequencies,
\[
P\ge\frac18\left(\frac23+S\right).
\]
In particular, $P\le1/33$ implies $S\le-14/33$.
\end{corollary}

\section{Pairs of short ratio}\label{sec:bernoulli}

We compute $P$ for two $3$-pairs and for a $3$-pair together with a pair
of ratio $7$ or $11$. These are the ratios corresponding to the three
most negative pair correlations, and the formulas will be used in both
parity patterns.
The common-sign set of a $3$-pair is a single arc modulo its period,
so the calculations reduce to correlations of interval indicators.

Let $B(t)=B_2(\{t\})$, where
\begin{equation}\label{eq:bernoulli}
B_2(u)=u^2-u+\frac16,
\qquad
\sum_{k\ge1}\frac{\cos(2\pi kt)}{k^2}=\pi^2B(t).
\end{equation}
For pairs in ratio $3$, we need only the values
\begin{equation}\label{eq:Btable}
\renewcommand{\arraystretch}{1.2}
\begin{array}{c*{6}{c}}
\toprule
u&0&1/6&1/3&1/2&2/3&5/6\\
\midrule
36B(u)&6&1&-2&-3&-2&1\\
\bottomrule
\end{array}
\end{equation}
Let $S_0(y)=\one\{\cos y>1/2\}$. Then
\[
\widehat S_0(0)=\E S_0=\frac13,
\qquad
\widehat S_0(k)=\frac{\sin(\pi k/3)}{\pi k},\quad k\ne0.
\]
Since
\[
\cos y\cos3y=\cos^2y\,(4\cos^2y-3)>0
\quad\Longleftrightarrow\quad \cos2y>\frac12
\]
almost everywhere, we have
\begin{equation}\label{eq:s3}
s(y)s(3y)=2S_0(2y)-1.
\end{equation}
Thus $S_0(2gx)$ is the common-sign indicator of $(g,3g)$.
We also have
\[
\one\{\cos y>0,\cos3y>0\}=\one\{|y|<\pi/6\},
\quad|y|<\pi.
\]

\subsection{Two pairs in ratio three}

\begin{proposition}\label{prop:doubletriple}
Let $g,h\in\N$ be coprime, with $g,3g,h,3h$ distinct.
\begin{enumerate}[label=\textup{(\alph*)}]
\item If $g,h$ are odd, then
\begin{equation}\label{eq:dt-odd}
P(g,3g,h,3h)
=\frac1{18}+\frac2{gh}
\left[B\left(\frac{g-h}{12}\right)-B\left(\frac{g+h}{12}\right)\right].
\end{equation}
Moreover, $P(g,3g,h,3h)\ge1/33$, with equality precisely when
$\{g,h\}=\{1,11\}$.
\item If $g,h$ have different parity, then
\begin{equation}\label{eq:dt-mixed}
P(g,3g,h,3h)
=\frac1{18}+\frac1{2gh}
\left[B\left(\frac{g-h}{6}\right)-B\left(\frac{g+h}{6}\right)\right].
\end{equation}
Moreover, $P(g,3g,h,3h)\ge1/24$, with equality precisely when
$\{g,h\}=\{1,4\}$.
\end{enumerate}
\end{proposition}

\begin{proof}
(a) Let $A(y)=\one\{|y|<\pi/6\}$ on $\R/2\pi\Z$.
By \eqref{eq:odd-symmetry}, $P=2\E[A(gx)A(hx)]$.
Since
\[
\widehat A(0)=\frac16,
\qquad
\widehat A(k)=\frac{\sin(\pi k/6)}{\pi k},\quad k\ne0,
\]
Parseval's identity and \eqref{eq:bernoulli} give
\[
\begin{aligned}
\E[A(gx)A(hx)]
&=\frac1{36}+\sum_{k\ne0}
\frac{\sin(\pi hk/6)\sin(\pi gk/6)}{\pi^2ghk^2}\\
&=\frac1{36}+\frac1{gh}
\left[B\left(\frac{g-h}{12}\right)-B\left(\frac{g+h}{12}\right)\right],
\end{aligned}
\]
which proves \eqref{eq:dt-odd}.

Let
\[
X=\frac{g-h}{2},\qquad Y=\frac{g+h}{2},\qquad
\Delta=B(X/6)-B(Y/6).
\]
The integers $X,Y$ have opposite parity. If $3\mid gh$, their residues
modulo $6$ are $\pm1$ and $\pm2$. Distinctness gives $gh\ge9$, so
\eqref{eq:Btable} yields
\[
\Delta\in\{-3/36,3/36\},
\qquad
P\ge\frac1{18}-\frac1{6gh}\ge\frac1{27}.
\]
If $3\nmid gh$, exactly one of $X,Y$ is divisible by $3$. The residue
pairs modulo $6$ are therefore $\{0,\pm1\}$ or $\{3,\pm2\}$, giving
\[
\Delta\in\{\pm5/36,\pm1/36\},
\qquad
P\ge\frac1{18}-\frac5{18gh}\ge\frac1{33},\quad gh\ge11.
\]
Equality requires $gh=11$. For $gh<11$, the only possible products are
$5$ and $7$, and \eqref{eq:dt-odd} gives
\[
P(1,3,5,15)=\frac1{15},\qquad P(1,3,7,21)=\frac1{21}.
\]
For $\{g,h\}=\{1,11\}$, we have $X\equiv\pm1$ and $Y\equiv0\pmod6$,
hence
\[
\Delta=-\frac5{36},\qquad
P=\frac1{18}-\frac5{198}=\frac1{33}.
\]

(b) Equations~\eqref{eq:parity-split} and \eqref{eq:s3} give
\[
P=\frac12\E[S_0(2gx)S_0(2hx)]
=\frac12\E[S_0(gy)S_0(hy)].
\]
Parseval's identity proves \eqref{eq:dt-mixed}. Since $g\pm h$ are odd,
\eqref{eq:Btable} gives
\[
P=
\begin{cases}
\dfrac1{18},&3\mid gh,\\[8pt]
\dfrac1{18}\pm\dfrac1{18gh},&3\nmid gh.
\end{cases}
\]
In the second case, $g-h,g+h$ have residues $3$ and $\pm1$ modulo $6$.
Thus
\[
P\ge\frac1{18}-\frac1{18gh}\ge\frac1{24},\quad gh\ge4.
\]
Equality requires $gh=4$ and the negative sign, which follows from
$B(1/2)-B(5/6)=-4/36$. The remaining product $gh=2$ gives
$P(1,2,3,6)=1/12$.
\end{proof}

\subsection{The kernels for ratios seven and eleven}

For $v\in\{7,11\}$, define
\[
E_v(y)=\one\{s(y/2)=s(vy/2)\},\qquad
m_v=\E E_v=\frac{v-1}{2v},\qquad
\beta_v=\frac{m_v}{6}.
\]
The function $E_v$ is $2\pi$-periodic, and $E_v(2hx)$ is the
common-sign indicator of $(h,hv)$. The value of $m_v$ follows from
Lemma~\ref{lem:pair}; in particular,
\[
\beta_7=\frac1{14},\qquad \beta_{11}=\frac5{66}.
\]
Since $s(y/2)=1$ on $(-\pi,\pi)$ and $v\equiv3\pmod4$,
\begin{equation}\label{eq:Earcs}
E_v=\sum_{|j|\le J}\one_{\{|y-4\pi j/v|<\pi/v\}},
\qquad J=\frac{v-3}{4},
\end{equation}
where the arcs are viewed in $\R/2\pi\Z$. Consequently, for $n\ne0$,
\begin{equation}\label{eq:Ev-fourier}
\widehat E_v(n)
=\frac{\sin(\pi n/v)}{\pi n}
  \sum_{|j|\le J}e^{-4\pi ijn/v}
=\frac{(-1)^{n+1}}{2\pi n}\tan\frac{\pi n}{v}.
\end{equation}
For $v\nmid n$, the last equality follows from
\[
\sum_{|j|\le J}e^{-4\pi ijn/v}
=\frac{\sin((2J+1)2\pi n/v)}{\sin(2\pi n/v)}
=\frac{(-1)^{n+1}}{2\cos(\pi n/v)};
\]
for $v\mid n$, both sides of \eqref{eq:Ev-fourier} vanish.
We shall also use the shift identity
\begin{equation}\label{eq:Ev-shift}
(-1)^n\widehat E_v(gn)
=\frac{g+v}{g}\widehat E_v((g+v)n),
\quad g\in2\N-1,\ n\in\Z\setminus\{0\},
\end{equation}
which follows from \eqref{eq:Ev-fourier}, since $g+v$ is even and
$\tan(\pi(g+v)n/v)=\tan(\pi gn/v)$.

For the correlation between $S_0(r\,\cdot)$ and $E_v$, put
\begin{equation}\label{eq:Kdef}
K_v(0)=0,\qquad
K_v(r)=\frac r2\left(\E[S_0(ry)E_v(y)]-\frac{m_v}{3}\right),
\quad r\in\N.
\end{equation}
The next lemma gives the period and values of $K_v$.

\begin{lemma}\label{lem:Kv}
For $v\in\{7,11\}$ and $r\in\N$,
\begin{equation}\label{eq:Kclosed}
K_v(r)=\frac12\sum_{j=-J}^{J}
\left[B\left(\frac16+\frac{(4j-1)r}{2v}\right)
-B\left(\frac16+\frac{(4j+1)r}{2v}\right)\right].
\end{equation}
On the nonnegative integers, $K_v$ satisfies
\begin{equation}\label{eq:Ksymmetry}
K_v(r+2v)=K_v(r),\qquad K_v(v)=0,\qquad
K_v(2v-r)=-K_v(r),\quad 0\le r\le2v.
\end{equation}
Its remaining values are
\[
\renewcommand{\arraystretch}{1.2}
\begin{array}{c*{10}{c}}
\toprule
r&1&2&3&4&5&6&7&8&9&10\\
\midrule
21K_7(r)&0&0&4&4&-2&1&0&&&\\
33K_{11}(r)&0&-2&4&-2&10&10&-4&-1&0&0\\
\bottomrule
\end{array}
\]
In particular,
\begin{equation}\label{eq:Kbounds}
|K_7(r)|\le\frac4{21},\qquad |K_{11}(r)|\le\frac{10}{33}.
\end{equation}
\end{lemma}

\begin{proof}
Parseval's identity and \eqref{eq:Earcs} give
\[
\begin{aligned}
\E[S_0(ry)E_v(y)]-\frac{m_v}{3}
&=\sum_{n\ne0}\widehat S_0(n)\widehat E_v(rn)\\
&=\frac1r\sum_{j=-J}^{J}
\left[B\left(\frac16+\frac{(4j-1)r}{2v}\right)
-B\left(\frac16+\frac{(4j+1)r}{2v}\right)\right].
\end{aligned}
\]
For the second equality, pair $j$ with $-j$, use
$2\sin\alpha\sin\beta=\cos(\alpha-\beta)-\cos(\alpha+\beta)$,
and apply \eqref{eq:bernoulli}. Multiplication by $r/2$ proves
\eqref{eq:Kclosed}.

By \eqref{eq:Ev-fourier},
\[
r\widehat E_v(rn)
=\frac{(-1)^{rn+1}}{2\pi n}\tan\frac{\pi rn}{v}.
\]
The right-hand side is $2v$-periodic and odd in $r$, and it vanishes
at $r=v$. This proves \eqref{eq:Ksymmetry}. Substituting in
\eqref{eq:Kclosed} gives the table and \eqref{eq:Kbounds}.
\end{proof}

For $h\in\N$, let
\[
T_hf(y)=\frac1h\sum_{j=0}^{h-1}f\left(\frac{y+2\pi j}{h}\right).
\]
Since $\widehat{T_hf}(n)=\widehat f(hn)$ and
$\widehat S_0(hn)=\sin(\pi hn/3)/(\pi hn)$ for $n\ne0$,
\begin{equation}\label{eq:ThS0}
T_hS_0=
\begin{cases}
\dfrac13,&3\mid h,\\[8pt]
\dfrac13+\dfrac{\chi_6(h)}h\left(S_0-\dfrac13\right),
&h\text{ odd},\ 3\nmid h,\\[8pt]
\dfrac13+\dfrac{\sigma(h)}h\left(S_0(\cdot+\pi)-\dfrac13\right),
&h\text{ even},\ 3\nmid h,
\end{cases}
\end{equation}
where $\sigma(h)=-1$ for $h\equiv2\pmod6$ and $\sigma(h)=1$ for
$h\equiv4\pmod6$. Indeed,
\[
\sin\frac{\pi hn}{3}=
\begin{cases}
\chi_6(h)\sin(\pi n/3),&h\text{ odd},\ 3\nmid h,\\[3pt]
-(-1)^n\sin(\pi n/3),&h\equiv2\pmod6,\\[3pt]
(-1)^n\sin(\pi n/3),&h\equiv4\pmod6,
\end{cases}
\]
and multiplication of Fourier coefficients by $(-1)^n$ translates
$S_0$ by $\pi$.

\begin{proposition}\label{prop:shortmixed}
Let $v\in\{3,7,11\}$, and let $g,h$ be coprime positive integers of
different parity. Then
\[
P(g,3g,h,hv)\ge\frac1{24},
\]
with equality if and only if
$\{g,3g,h,hv\}=\{1,3,4,12\}$. For $v\in\{7,11\}$, the exact formula is
\begin{equation}\label{eq:shortmixed}
P(g,3g,h,hv)=
\begin{cases}
\beta_v,&3\mid h,\\[8pt]
\beta_v+\dfrac{\chi_6(h)K_v(g)}{gh},&h\text{ odd},\ 3\nmid h,\\[8pt]
\beta_v+\dfrac{\sigma(h)K_v(g+v)}{gh},&h\text{ even},\ 3\nmid h.
\end{cases}
\end{equation}
\end{proposition}

\begin{proof}
The case $v=3$ is Proposition~\ref{prop:doubletriple}(b). Assume henceforth
that $v\in\{7,11\}$.
The frequencies are distinct by parity. Equations~\eqref{eq:parity-split}
and \eqref{eq:s3} give
\[
P=\frac12\E[S_0(2gx)E_v(2hx)]
=\frac12\E[S_0(gy)E_v(hy)].
\]
Condition on $z=hy$. Since $\gcd(g,h)=1$, multiplication by $g$
permutes the preimages, and
\[
P=\frac12\E[E_v(z)(T_hS_0)(gz)].
\]
If $3\mid h$, \eqref{eq:ThS0} gives $P=\beta_v$.
For odd $h$ with $3\nmid h$, it gives
\[
P=\frac{m_v}{6}+\frac{\chi_6(h)}{2h}
\left(\E[E_v(z)S_0(gz)]-\frac{m_v}{3}\right)
=\beta_v+\frac{\chi_6(h)K_v(g)}{gh}.
\]
For even $h$ with $3\nmid h$, $g$ is odd, and
\eqref{eq:Ev-shift} gives
\[
\begin{aligned}
\E[E_v(z)S_0(gz+\pi)]-\frac{m_v}{3}
&=\sum_{n\ne0}\widehat S_0(n)(-1)^n\widehat E_v(gn)\\
&=\frac{g+v}{g}\sum_{n\ne0}
\widehat S_0(n)\widehat E_v((g+v)n)\\
&=\frac{2K_v(g+v)}{g}.
\end{aligned}
\]
Together with \eqref{eq:ThS0}, this proves \eqref{eq:shortmixed}.

By \eqref{eq:Kbounds},
\[
P\ge\beta_v-\frac{M_v}{gh},
\qquad M_7=\frac4{21},\quad M_{11}=\frac{10}{33}.
\]
Thus $P\ge1/21$ for $v=7$, $gh\ge8$, and $P\ge1/22$ for $v=11$,
$gh\ge10$. The remaining coprime pairs of different parity have
$gh\in\{2,4,6\}$, or $gh=8$ when $v=11$. Formula
\eqref{eq:shortmixed} gives the following values, all greater than $1/24$:
\[
\renewcommand{\arraystretch}{1.7}
\begin{array}{ccc@{\qquad}ccc}
\toprule
(g,h)&v=7&v=11&(g,h)&v=7&v=11\\
\midrule
(1,2)&\dfrac2{21}&\dfrac5{66}&(2,1)&\dfrac1{14}&\dfrac1{22}\\[4pt]
(1,4)&\dfrac5{84}&\dfrac5{66}&(4,1)&\dfrac5{42}&\dfrac2{33}\\[4pt]
(1,6)&\dfrac1{14}&\dfrac5{66}&(6,1)&\dfrac5{63}&\dfrac{25}{198}\\[4pt]
(2,3)&\dfrac1{14}&\dfrac5{66}&(3,2)&\dfrac{13}{126}&\dfrac7{99}\\[4pt]
(1,8)&\cdot&\dfrac5{66}&(8,1)&\cdot&\dfrac{19}{264}\\
\bottomrule
\end{array}
\]
\end{proof}

\begin{proposition}\label{prop:shortodd}
Let $v\in\{7,11\}$, and let $g,h$ be odd positive integers with
\[
\gcd(g,h)=\gcd(g,v)=1,\qquad 3\nmid h,
\]
and $g,3g,h,hv$ distinct. Then
\begin{equation}\label{eq:shortodd}
P(g,3g,h,hv)
=\beta_v+\frac{\chi_6(h)K_v(g)+\chi(gh)\beta_v}{gh}>\frac1{33}.
\end{equation}
\end{proposition}

\begin{proof}
For $\epsilon_1,\ldots,\epsilon_4\in\{\pm1\}$, let
\[
e_A=\one[\epsilon_1=\epsilon_2],\qquad
f_A=\frac{\epsilon_1+\epsilon_2}{2},
\]
and define $e_B,f_B$ similarly using $\epsilon_3,\epsilon_4$. Then
\[
\one[\epsilon_1=\epsilon_2=\epsilon_3=\epsilon_4]
=\frac12e_Ae_B+\frac12f_Af_B.
\]
Apply this identity to
\[
(\epsilon_1,\epsilon_2,\epsilon_3,\epsilon_4)
=(s(gx),s(3gx),s(hx),s(hvx)).
\]
As in the proof of Proposition~\ref{prop:shortmixed}, the first term has
expectation
\[
\frac12\E[S_0(gy)E_v(hy)]
=\beta_v+\frac{\chi_6(h)K_v(g)}{gh}.
\]
For the second term, each cross-pair is coprime by the hypotheses and
$3\nmid v$. Lemma~\ref{lem:pair} therefore gives
\[
\begin{aligned}
\frac18\bigl[C(g,h)+C(g,hv)+C(3g,h)+C(3g,hv)\bigr]
&=\frac{\chi(gh)}{8gh}
\left(1-\frac1v-\frac13+\frac1{3v}\right)=\frac{\chi(gh)\beta_v}{gh},
\end{aligned}
\]
proving \eqref{eq:shortodd}. By \eqref{eq:Kbounds},
\[
\begin{aligned}
P(g,3g,h,7h)&\ge\frac1{14}-\frac{11}{42gh},\\
P(g,3g,h,11h)&\ge\frac5{66}-\frac{25}{66gh}.
\end{aligned}
\]
These exceed $1/33$ for $gh\ge7$ and $gh\ge9$, respectively. The
remaining pairs are $(3,1),(1,5),(5,1)$, and, for $v=11$,
$(1,7),(7,1)$. Formula~\eqref{eq:shortodd} gives
\[
\renewcommand{\arraystretch}{1.6}
\begin{array}{c*{5}{c}}
\toprule
(g,h)&(3,1)&(1,5)&(5,1)&(1,7)&(7,1)\\
\midrule
v=7&\dfrac19&\dfrac3{35}&\dfrac1{15}&\cdot&\cdot\\[6pt]
v=11&\dfrac1{11}&\dfrac1{11}&\dfrac5{33}&\dfrac5{77}&\dfrac1{21}\\
\bottomrule
\end{array}
\]
all greater than $1/33$.
\end{proof}

\section{Odd frequencies}\label{sec:odd}

For an odd quadruple $A$, Corollary~\ref{cor:S} and
\eqref{eq:initial-arc} give
\begin{equation}\label{eq:lowconseq}
P(A)\le\frac1{33}
\quad\Longrightarrow\quad
S(A)\le-\frac{14}{33},\qquad \max A\ge33.
\end{equation}
Together with an interval calculation for one
exceptional configuration, the first constraint will force a $3$-pair
whenever $P(A)\le1/33$. With exactly one such pair, it constrains the
reduced ratios, while the component
estimate excludes all but finitely many choices of the remaining
frequencies.
Two disjoint $3$-pairs are already covered by
Proposition~\ref{prop:doubletriple}.

\subsection{No pair in ratio three}\label{sec:notriple}

\begin{proposition}\label{prop:notriple}
Let $a_1,\dots,a_4$ be distinct odd positive integers. If no two form a
$3$-pair, then $P\ge29/924>1/33$.
\end{proposition}

\begin{proof}
Join two frequencies when their correlation is negative. This graph is
bipartite between the residue classes $1$ and $3$ modulo $4$, hence has
at most four edges. By \eqref{eq:pairodd}, every edge has weight
$|C|\le1/7$, with equality only for ratio $1{:}7$; call it a
\emph{$7$-edge}. All other weights are at most $1/11$.
If there is at most one $7$-edge, Corollary~\ref{cor:S} gives
\[
S\ge-\frac17-\frac3{11}=-\frac{32}{77},\qquad
P\ge\frac{29}{924}.
\]
Otherwise choose two $7$-edges.

\emph{Case 1: two $7$-edges share a vertex.}
Let the vertices be $a,7a,49a,b$. The first three contribute
\[
-\frac17-\frac17+\frac1{49}=-\frac{13}{49}.
\]
If $\chi(b)=\chi(a)$, the only other possible negative edge is
$\{b,7a\}$. Thus $S\ge-20/49$ and $P\ge19/588$.

Suppose $\chi(b)=-\chi(a)$.
Write $a=dv$ and $b=du$, where $\gcd(u,v)=1$.
By \eqref{eq:pairodd}, the remaining correlations sum to
\[
R=\frac{-1+\gcd(u,7)^2/7-\gcd(u,49)^2/49}{uv}.
\]
The absence of a $3$-pair gives $uv\ge7$, and also $uv/49\ge7$
when $49\mid u$, by considering the reduced ratio $(u/49){:}v$ of
$b$ and $49a$. Therefore
\[
R=
\begin{cases}
-\dfrac{43}{49uv}\ge-\dfrac{43}{343},&7\nmid u,\\[6pt]
\dfrac5{uv}>0,&7\mid u,\ 49\nmid u,\\[6pt]
-\dfrac{43}{uv}\ge-\dfrac{43}{343},&49\mid u,
\end{cases}
\]
and
\[
S\ge-\frac{13}{49}-\frac{43}{343}=-\frac{134}{343},\qquad
P\ge\frac18\left(\frac23-\frac{134}{343}\right)=\frac{71}{2058}.
\]

\emph{Case 2: two disjoint $7$-edges.}
The configuration is $(a,7a,b,7b)$.
Write $a=dv$ and $b=du$, with $\gcd(u,v)=1$, and let $h=\chi(uv)$.
Equation \eqref{eq:pairodd} gives
\[
S=-\frac27+R,\qquad
R=\frac{h}{uv}\left(2-\frac{\gcd(u,7)^2+\gcd(v,7)^2}{7}\right).
\]
Suppose first that $7\nmid uv$, so $R=12h/(7uv)$.
If $h=1$, then $S>-2/7$.
If $h=-1$, then $uv\equiv3\pmod 4$, while $uv\ne3$ because there is no
$3$-pair, and $uv\ne7$ by assumption.
For $uv\ge15$,
\[
S\ge-\frac27-\frac{12}{105}=-\frac25,\qquad P\ge\frac1{30}.
\]
The remaining value $uv=11$ gives $d\{1,7,11,77\}$. By dilation,
take $d=1$. The common positive set in $t\in[0,1/2]$ is
\[
\left(0,\frac1{308}\right)
\cup\left(\frac3{308},\frac5{308}\right)
\cup\left(\frac{51}{308},\frac{53}{308}\right).
\]
Equation~\eqref{eq:odd-symmetry} gives $P=4\cdot5/308=5/77$.

Suppose now that $7\mid uv$.
The bracket equals $-36/7$, so $R>0$ if $h=-1$.
For $h=1$, write $uv=7w$, where $w\equiv3\pmod 4$.
The value $w=3$ forces $\{u,v\}=\{1,21\}$ or $\{3,7\}$, both of which
produce a $3$-pair.
Hence $uv\ge49$, giving
$S\ge-2/7-36/343=-134/343$ and
$P\ge71/2058$.

In every case with at least two $7$-edges, the lower bound exceeds
$29/924$.
\end{proof}

\subsection{Eliminating a fourth frequency}\label{sec:odd-extension}

The correlation estimates bound the reduced ratios, but may leave the largest
frequency unbounded. For large $c$, the component estimate gives a lower
bound for $P(T,c)$ close to $P(T)/2$. Combining this with the initial arc
bound gives a lower bound independent of $c$.

\begin{lemma}\label{lem:triple-elimination}
Let $T$ be a triple of distinct odd positive integers, let $c$ be an odd
positive integer outside $T$, and let $M\ge N(T)$. Then
\begin{equation}\label{eq:triple-uniform}
P(T,c)\ge
\min\left\{\frac1{\max T},\frac{2P(T)}{M+4}\right\}.
\end{equation}
More precisely, if $p>0$, $P(T,c)\le p$, and $P(T)>2p$, then
\begin{equation}\label{eq:fourth-cutoff}
c\le\frac{M}{2(P(T)-2p)}.
\end{equation}
If also $\max T<1/p$, then $c\ge1/p$.
\end{lemma}

\begin{proof}
For $c<\max T$, \eqref{eq:initial-arc} gives $P(T,c)\ge1/\max T$.
For $c>\max T$, \eqref{eq:component} and \eqref{eq:initial-arc} give
\[
P(T,c)\ge\max\left\{\frac12P(T)-\frac{M}{4c},\frac1c\right\}.
\]
In particular,
\[
\frac12P(T)\le P(T,c)+\frac{M}{4c}
\le\left(1+\frac M4\right)P(T,c),
\]
which proves \eqref{eq:triple-uniform}.
If $P(T,c)\le p$, the component estimate also gives
\[
\frac12P(T)-\frac{M}{4c}\le p,
\]
and rearrangement proves \eqref{eq:fourth-cutoff}. Finally,
$\max(T\cup\{c\})\ge1/p$ by \eqref{eq:initial-arc}, giving the
last assertion.
\end{proof}

We apply the criterion to the small triples that arise in the ratio
classification. Define
\[
\begin{aligned}
\mathcal T_0=\bigl\{&\{1,3,5\},\{1,3,7\},\{1,3,9\},\{1,3,19\},\{3,5,9\},\{3,5,11\},\{1,11,23\}\bigr\}.
\end{aligned}
\]

The interval calculations used in the next lemma are given in
Appendix~\ref{sec:odd-certificates}.

\begin{lemma}\label{lem:triples}
Let $A$ be a quadruple of distinct odd positive integers. Then
\begin{equation}\label{eq:triple-exclusion}
T\subset A\text{ for some }T\in\mathcal T_0
\quad\Longrightarrow\quad P(A)>\frac1{33}.
\end{equation}
If $P(A)\le1/33$ and $A=T\cup\{c\}$, then
\begin{equation}\label{eq:triple-fourth}
\begin{aligned}
T=\{1,3,11\}&\quad\Longrightarrow\quad c=33,\\
T\in\bigl\{\{1,3,21\},\{1,7,21\},\{3,7,21\}\bigr\}
&\quad\Longrightarrow\quad 33\le c\le36,\\
T=\{1,3,33\}&\quad\Longrightarrow\quad c\le55.
\end{aligned}
\end{equation}
\end{lemma}

\begin{proof}
For each $T\in\mathcal T_0$, Table~\ref{tab:odd-triples} gives
\[
\max T\le23,\qquad
N(T)\le M(T),\qquad \frac{2P(T)}{M(T)+4}>\frac1{33}.
\]
Thus \eqref{eq:triple-uniform} proves \eqref{eq:triple-exclusion}.
For the triples in \eqref{eq:triple-fourth}, the same table and
\eqref{eq:fourth-cutoff}, with $p=1/33$, give respectively
\[
\begin{aligned}
c&\le\frac{6}{2(5/33-2/33)}=33,\\
c&\le\frac{6}{2(1/7-2/33)}=\frac{693}{19}<37,\\
c&\le\frac{10}{2(5/33-2/33)}=55.
\end{aligned}
\]
For the first two lines, $\max T<33$, so the last assertion of
Lemma~\ref{lem:triple-elimination} gives $c\ge33$.
\end{proof}

\subsection{Reduction to finitely many quadruples}\label{sec:onetriple}

The pair correlation bounds and Lemma~\ref{lem:triples} reduce the
one-pair case to the following ten quadruples.

We identify an increasing quadruple with its underlying frequency set and put
\begin{equation}\label{eq:A35}
\mathcal E_{\mathrm I}
=\bigl\{(1,11,19,33),\ (1,11,31,33),\ (1,15,19,45)\bigr\},
\end{equation}
\begin{equation}\label{eq:A49}
\begin{aligned}
\mathcal E_{\mathrm {II}}=\bigl\{&(1,3,21,35),\ (1,3,23,33),\ (1,3,33,55),\\
&(1,11,15,33),\ (1,11,27,33),\ (1,11,33,39),\ (3,11,13,33)\bigr\}.
\end{aligned}
\end{equation}
Write $\mathcal E=\mathcal E_{\mathrm I}\cup\mathcal E_{\mathrm {II}}$.

\begin{proposition}\label{prop:odd-reduction}
Let $A$ be a primitive quadruple of distinct odd positive integers
containing exactly one $3$-pair. Then
\begin{equation}\label{eq:odd-finite-reduction}
P(A)\le\frac1{33}\quad\Longrightarrow\quad A\in\mathcal E.
\end{equation}
\end{proposition}

The proof is given in Appendix~\ref{sec:odd-arithmetic}. Its two classes
correspond to whether neither or exactly one of the remaining frequencies
is divisible by $3$, yielding $\mathcal E_{\mathrm I}$ and
$\mathcal E_{\mathrm {II}}$, respectively. The following direct interval
check completes the argument.

\begin{lemma}\label{lem:odd-candidates}
For every $A\in\mathcal E$,
\begin{equation}\label{eq:odd-candidate-bound}
P(A)\ge\frac{28}{855}>\frac1{33}.
\end{equation}
\end{lemma}

\begin{proof}
For every $A\in\mathcal E$, Table~\ref{tab:odd-candidates} gives disjoint
intervals in $[0,1/2]$ on which all four cosines are positive, with
$4L\ge28/855$, where $L$ is their total length.
By reflection and \eqref{eq:odd-symmetry}, $P(A)\ge4L$.
The endpoint and sign verifications appear in
Appendix~\ref{app:odd-candidate-certificates}.
\end{proof}

\begin{proposition}\label{thm:onetriple}
If $A$ is a primitive quadruple of distinct odd positive integers
containing exactly one $3$-pair, then
\[
P(A)>\frac1{33}.
\]
\end{proposition}

\begin{proof}
If $P(A)\le1/33$, Proposition~\ref{prop:odd-reduction} gives
$A\in\mathcal E$, contrary to Lemma~\ref{lem:odd-candidates}.
\end{proof}

\section{Conditional correlations}\label{sec:auxiliary}

For an odd $3$-pair, parity splitting reduces the common-sign
probability to a conditional two-frequency correlation. A sharp lower
bound for this correlation will also determine the equality cases.

Let $a$ be an odd positive integer, and let $b,c$ be distinct positive
integers with $\gcd(a,bc)=1$. Conditioning on the map $x\mapsto ax$
leads to the sampled correlation
\begin{equation}\label{eq:sampled-correlation}
F_a(\tau)=\frac1a\sum_{j=0}^{a-1}
s\left(\frac{b(\tau+2\pi j)}a\right)
s\left(\frac{c(\tau+2\pi j)}a\right).
\end{equation}
We shall average it over the central arc:
\begin{equation}\label{eq:sampled-mean}
\rho=\frac3{2\pi}\int_{-\pi/3}^{\pi/3}F_a(\tau)\,d\tau.
\end{equation}
The parity identity \eqref{eq:parity-split} and the change of variables
$y=2x$ give
\begin{equation}\label{eq:oddtriple-conditioning}
\begin{aligned}
P(a,3a,2b,2c)
&=\frac14\E\bigl[S_0(ay)(1+s(by)s(cy))\bigr]=\frac{1+\rho}{12}.
\end{aligned}
\end{equation}
For the second equality, condition on $ay$: the average of
$s(by)s(cy)$ on each fiber is $F_a$, and
$S_0(\tau)=\one_{\{|\tau|<\pi/3\}}$ on $[-\pi,\pi)$.
Thus $P\ge1/24$ is equivalent to $\rho\ge-1/2$.

\begin{proposition}\label{lem:sampled-half}
Let $a$ be an odd positive integer and let $b,c$ be distinct positive
integers with $\gcd(a,bc)=1$. For the average \eqref{eq:sampled-mean},
\[
\rho\ge-\frac12,
\]
with equality if and only if
\[
(a,\{b,c\})\in\bigl\{(1,\{2,6\}),(5,\{1,2\})\bigr\}.
\]
\end{proposition}

For $a=1$, this is an initial-interval estimate. For $a\ge3$, we distinguish
the cases
\[
c\equiv\pm b\pmod a
\qquad\text{and}\qquad
c\not\equiv\pm b\pmod a.
\]
In the first case, one square wave is a translate of the other on the
grid. In the second, the positive sets are affine images of discrete
half-circles. The corresponding estimates give strict inequality
except for the five-point grid, which must be examined separately.
We prove the proposition in Section~\ref{sec:conditional-proof} after
proving these estimates. The quadrant estimate at the end of the
section will be used for an even $3$-pair.

\Needspace{20\baselineskip}
\subsection{Initial intervals}\label{sec:prefix}

We first prove a bound for real frequencies, also needed for the
$7$-pair estimates in Section~\ref{sec:mixed}. For integer frequencies
on $[0,\pi/3]$, we then sharpen it to the bound required here.

Write $\sigma(x)=\sgn\cos(\pi x)$ and, for $r,L>0$, set
\begin{equation}\label{eq:prefix-agreement}
\begin{gathered}
I_r(L)=\int_0^L\sigma(x)\sigma(rx)\,dx,\\
A_r(L)=\bigl|\{x\in[0,L]:\sigma(x)=\sigma(rx)\}\bigr|,
\qquad I_r(L)=2A_r(L)-L.
\end{gathered}
\end{equation}
\begin{lemma}\label{thm:prefix}
For all positive real numbers $u,v,T$,
\[
\int_0^T\sgn\cos(ut)\,\sgn\cos(vt)\,dt\ge-\frac35T.
\]
Equality holds when $v=3u$ and $T=5\pi/(6u)$, so the constant is sharp.
\end{lemma}

\begin{proof}
By symmetry and the substitution $x=ut/\pi$, it suffices to prove
\begin{equation}\label{eq:F11}
I_r(L)\ge-\frac35L
\quad\Longleftrightarrow\quad A_r(L)\ge\frac L5,
\quad r\ge1,\ L>0.
\end{equation}
The case $r=1$ is immediate; assume $r>1$.

We begin with two elementary estimates, which will also be used in
the next lemma.
If $J$ is an interval of length $\ell$ and $\epsilon\in\{\pm1\}$, then
\begin{equation}\label{eq:interval-sign-count}
\begin{aligned}
\bigl|\{x\in J:\sigma(rx)=\epsilon\}\bigr|
&\ge\max\{0,(\ell-r^{-1})/2\}, &&\ell>0,\\
\bigl|\{x\in J:\sigma(rx)=\epsilon\}\bigr|
&\ge\max\{0,\ell-r^{-1}\}, &&\ell\le2/r.
\end{aligned}
\end{equation}
The first estimate follows because a primitive of $\sigma(r\,\cdot)$
has range $1/r$; for the second, extend $J$ to a full period, in which
each sign occupies length $1/r$. On an initial interval, we also have
\begin{equation}\label{eq:prefix-sign-count}
\bigl|\{x\in[0,z]:\sigma(x)=1\}\bigr|\ge\frac z3,
\quad z\ge0.
\end{equation}
Indeed, writing $z=2n+\tau$, with $n\ge0$ and $0\le\tau<2$, the
left-hand side is $n+p(\tau)$, where
\[
p(\tau)=
\begin{cases}
\tau,&0\le\tau\le1/2,\\
1/2,&1/2\le\tau\le3/2,\\
\tau-1,&3/2\le\tau<2,
\end{cases}
\qquad 3p(\tau)\ge\tau.
\]
For $z>0$, equality in \eqref{eq:prefix-sign-count} holds exactly at
$z=3/2$. Both the estimate and its equality condition extend to
$\sigma(r\,\cdot)$ by scaling.

If $L\le1/2$, then $\sigma=1$ on $[0,L]$, and \eqref{eq:prefix-sign-count} gives $A_r(L)\ge L/3$.
Otherwise let $L=m-1/2+d$, where $m\in\N$ and $0\le d<1$.
The intervals of constancy of $\sigma$ partition $[0,L]$ into one interval
of length $1/2$, then $m-1$ intervals of length $1$, and one of length $d$.

\emph{Case 1: $r\ge3$.} Applying \eqref{eq:prefix-sign-count} on
the first interval and \eqref{eq:interval-sign-count} on the others gives
\[
A_r(L)\ge\frac16+\frac{m-1}3+\max\Bigl(0,\frac{d-1/3}2\Bigr).
\]
Hence
\[
\begin{aligned}
5A_r(L)-L
&\ge\frac{2m-1}3-d+\frac52\max\left(0,d-\frac13\right)\\
&\ge\frac{2(m-1)}3\ge0,
\end{aligned}
\]
where the second inequality is sharp at $d=1/3$.

\emph{Case 2: $2\le r\le3$.} The first interval contributes $1/(2r)$:
the next positive interval of $\sigma(rx)$ starts at $3/(2r)\ge1/2$.
The first bound in \eqref{eq:interval-sign-count} gives
\[
A_r(L)\ge\frac1{2r}+\frac{m-1}2\Bigl(1-\frac1r\Bigr)+\frac12\max\Bigl(0,d-\frac1r\Bigr).
\]
Minimizing over $d$ gives
\[
\begin{gathered}
5A_r(L)-L\ge\frac{(3m-4)r-5m+8}{2r},\\
(3m-4)r-5m+8\ge
\begin{cases}
3-r\ge0,&m=1,\\
m>0,&m\ge2.
\end{cases}
\end{gathered}
\]

\emph{Case 3: $1<r\le2$.} The first interval again contributes $1/(2r)$.
Each remaining interval has length at most $1\le2/r$, so the second
bound in \eqref{eq:interval-sign-count} gives
\[
A_r(L)\ge\frac1{2r}+(m-1)\Bigl(1-\frac1r\Bigr)+\max\Bigl(0,d-\frac1r\Bigr).
\]
Minimizing over $d$ gives
\[
5A_r(L)-L\ge\frac{(8m-9)r+13-10m}{2r},
\]
where
\[
(8m-9)r+13-10m\ge
\begin{cases}
3-r>0,&m=1,\\[6pt]
7(r-1)>0,&m=2,\\[6pt]
\dfrac74>0,&m\ge3,\ r\ge\dfrac54.
\end{cases}
\]

It remains to consider $1<r<5/4$ and $m\ge3$. Let $\varepsilon=r-1$ and
\[
G(z)=\int_0^z\sigma(x)\,dx,\qquad
C(y)=\frac12-\dist(y,2\Z).
\]
Both functions are $2$-periodic, $|G|\le1/2$, and $C$ is $1$-Lipschitz.
Moreover,
\begin{equation}\label{eq:F52}
G\left(j+\frac12+y\right)=(-1)^jC(y),
\quad j\in\Z,\ y\in\R.
\end{equation}
The jump of $\sigma$ at $j+1/2$ is $2(-1)^{j+1}$. Integration by parts gives
\begin{equation}\label{eq:F53}
\begin{aligned}
I_r(L)&=\frac{\sigma(L)G(rL)}r
+\frac2r\sum_{j=0}^{m-1}(-1)^jG\left(r\left(j+\frac12\right)\right)\\
&=\frac{\sigma(L)G(rL)}r
+\frac2r\sum_{j=0}^{m-1}C\left(\varepsilon\left(j+\frac12\right)\right),
\end{aligned}
\end{equation}
where the second equality follows from \eqref{eq:F52}.
At $L=m-1/2$, take the right-hand value of $\sigma(L)$; the identity
extends by continuity.

Let
\[
H(z)=\int_0^zC(y)\,dy=
\begin{cases}
\dfrac{z(1-z)}2,&0\le z\le1,\\[6pt]
\dfrac{(z-1)(z-2)}2,&1\le z\le2.
\end{cases}
\]
On $[0,1]$, $H\ge0$; on $[1,2]$,
\[
10H(z)+z=5\left(z-\frac75\right)^2+\frac15>0.
\]
Since $H$ is $2$-periodic,
\begin{equation}\label{eq:F55}
H(z)\ge-\frac z{10},\qquad z\ge0.
\end{equation}
The Lipschitz bound gives
\[
\int_{j\varepsilon}^{(j+1)\varepsilon}C(y)\,dy
\le\varepsilon C\left(\varepsilon\left(j+\frac12\right)\right)
+\frac{\varepsilon^2}4.
\]
Summing and using \eqref{eq:F55}, we obtain
\[
\begin{aligned}
\sum_{j=0}^{m-1}C\left(\varepsilon\left(j+\frac12\right)\right)
\ge\frac{H(m\varepsilon)}\varepsilon-\frac{m\varepsilon}4
\ge-\frac m{10}-\frac m{16}=-\frac{13m}{80},
\end{aligned}
\]
since $\varepsilon<1/4$. Equation~\eqref{eq:F53} gives
$I_r(L)\ge-1/2-13m/40$.
Since
\[
\frac35\Bigl(m-\frac12\Bigr)-\frac12-\frac{13m}{40}
=\frac{11m-32}{40}>0,\quad m\ge3,
\]
we conclude that $I_r(L)\ge-3(m-1/2)/5\ge-3L/5$.

For sharpness, let $r=3$ and $L=5/6$. The waves agree on $(0,1/6)$
and disagree on $(1/6,5/6)$, so
\[
I_3(5/6)=\frac16-\frac23=-\frac12=-\frac35\cdot\frac56.
\]
\end{proof}

\Needspace{9\baselineskip}
For the fixed interval $[0,\pi/3]$, integrality of the frequencies
improves the constant and determines the equality case.

\begin{lemma}\label{lem:integerprefix}
For positive integers $b,c$,
\[
\frac3\pi\int_0^{\pi/3}s(bt)s(ct)\,dt\ge-\frac12.
\]
Equality holds if and only if $\{b,c\}=\{2,6\}$.
\end{lemma}

\begin{proof}
If $b=c$, the normalized integral is $1$. By symmetry, assume $b<c$.
In the notation of Lemma~\ref{thm:prefix}, let
$r=c/b$ and $L=b/3$. We prove $A_r(L)\ge L/4$.
For $b=1$, \eqref{eq:prefix-sign-count} gives $A_r(L)\ge L/3$.
For $b\ge2$, write
\[
L=m-\frac12+d,\qquad m\ge1,\qquad
 d\in\left\{\frac16,\frac12,\frac56\right\}.
\]

If $r\ge3$, \eqref{eq:interval-sign-count} and
\eqref{eq:prefix-sign-count} give
\[
4A_r(L)-L
\ge\frac m3-\frac16-d+2\max\left(0,d-\frac13\right)
\ge\frac{m-1}{3}.
\]
Equality requires $m=1$, $d\in\{1/6,1/2\}$, and equality in
\eqref{eq:prefix-sign-count} on the initial half interval. Its equality
condition gives $r/2=3/2$, hence
\[
(b,c)\in\{(2,6),(3,9)\}.
\]
The corresponding normalized integrals are $-1/2$ and $-1/3$.

If $2\le r\le3$, the first bound in \eqref{eq:interval-sign-count} gives
\[
4A_r(L)-L\ge m-\frac32+\frac{4-2m}{r}
-d+2\max\left(0,d-\frac1r\right).
\]
For $m\ge2$, the right-hand side is at least
\[
(2m-3)\left(\frac12-\frac1r\right)\ge0.
\]
Equality requires $r=2$ and $d=1/2$, but $I_2(m)=0$.
For $m=1$, the values $b=2,3,4$ give the respective lower bounds
\[
\frac2r-\frac23,\qquad 0,\qquad\frac13.
\]
The first vanishes only at $r=3$. For $b=3$, direct integration gives
\[
\frac3\pi\int_0^{\pi/3}s(3t)s(ct)\,dt
=
\begin{cases}
0,             & c=6,\\[3pt]
-2/7,     & c=7,\\[3pt]
-3/8,     & c=8,\\[3pt]
-1/3,     & c=9.
\end{cases}
\]

For $1<r<2$, the second bound in \eqref{eq:interval-sign-count} gives
\begin{equation}\label{eq:intprefixB}
\begin{aligned}
4A_r(L)-L
&\ge3m-\frac72+\frac{6-4m}{r}
-d+4\max\left(0,d-\frac1r\right)\\
&\ge3m-\frac72+\frac{5-4m}{r}.
\end{aligned}
\end{equation}
This is positive for $m=1$, and at least $1/4$ for
$m\ge2$, $r\ge4/3$.
For $1<r<4/3$ and $m\ge3$, let $\varepsilon=r-1$.
The integration-by-parts identity \eqref{eq:F53} and the Lipschitz
estimate following \eqref{eq:F55} give
\[
I_r(L)\ge-
\frac{1/2+m/5+m\varepsilon/2}{1+\varepsilon}
> -\frac12\left(m-\frac12\right)\ge-\frac L2,
\]
since
\[
\frac{3m}{10}-\frac34-\frac\varepsilon4>0,
\quad m\ge3,\quad 0<\varepsilon<1/3.
\]
For $m=2$, the remaining pairs are
$(b,c)=(5,6),(6,7),(7,8),(7,9)$.
The first line of \eqref{eq:intprefixB} gives the positive lower bounds
$2/3$, $2/7$, and $1/3$ for the first, second, and fourth pairs.
For the third,
\[
\frac3\pi\int_0^{\pi/3}s(7t)s(8t)\,dt=\frac{37}{56}.
\]
Thus equality occurs only for $\{b,c\}=\{2,6\}$.
\end{proof}

\subsection{Translated square waves on a grid}\label{sec:grid-correlation}

For a displacement $w\in\R$, the continuous correlation is
\begin{equation}\label{eq:triangular-correlation}
\Lambda(w):=\frac1\pi\int_0^\pi s(z)s(z+w)\,dz
=1-\frac2\pi\dist(w,2\pi\Z).
\end{equation}
Counting grid points in the arc where the product is negative gives
the following discrete bounds.

\begin{lemma}\label{lem:grid-correlation}
Let $N\in\N$ and $z,w\in\R$, and suppose that none of the cosines
sampled below vanishes. Then
\begin{equation}\label{eq:grid-correlation}
\begin{aligned}
\frac1N\sum_{j=0}^{N-1}
s\left(z+\frac{\pi j}N\right)
s\left(z+w+\frac{\pi j}N\right)
&\ge1-\frac2N\left\lceil
\frac N\pi\dist(w,2\pi\Z)\right\rceil\\
&\ge\Lambda(w)-\frac2N.
\end{aligned}
\end{equation}
\end{lemma}

\begin{proof}
The product $s(u)s(u+w)$ is $\pi$-periodic, and its negative set is
an arc of length $\delta=\dist(w,2\pi\Z)$ on $\R/\pi\Z$.
An arc of length $\delta$ contains at most $\lceil N\delta/\pi\rceil$
points of the grid, provided its endpoints are not sampled. Hence
the average is at least $1-2\lceil N\delta/\pi\rceil/N$.
The cases $\delta=0,\pi$ give the exact values $1,-1$; the second
inequality follows from $\lceil t\rceil\le t+1$.
\end{proof}

In particular, let $a,b,c$ satisfy the hypotheses of
Proposition~\ref{lem:sampled-half}, and suppose
$\epsilon c=b+ka$ for some $\epsilon\in\{\pm1\}$ and $k\in\Z$.
Then
\[
F_a(\tau)=\frac1a\sum_{j=0}^{a-1}s(z_j)s(z_j+k\tau),
\qquad z_j=\frac{b(\tau+2\pi j)}a.
\]
Since $\gcd(2b,a)=1$, these points form a uniform grid modulo $\pi$.
Thus, for almost every $\tau$,
\begin{equation}\label{eq:congruent-grid}
F_a(\tau)\ge
1-\frac2a\left\lceil\frac a\pi\dist(k\tau,2\pi\Z)\right\rceil
\ge\Lambda(k\tau)-\frac2a.
\end{equation}

\subsection{Half-circles in cyclic groups}

An interval in $\Z/N\Z$ is a set of consecutive residues. Both the odd and even
moduli will be treated using the following translation estimate. If $I,J$ are intervals,
$|I|\in\{\lfloor N/2\rfloor,\lceil N/2\rceil\}$, and
$\mathcal J=rJ+d$, where $r$ is a unit modulo $N$, then
\[
\begin{aligned}
2\dist(kr,N\Z)
&=|I\symd(I+kr)|\\
&\le2|I\symd\mathcal J|+|\mathcal J\symd(\mathcal J+kr)|\\
&=2|I\symd\mathcal J|+|J\symd(J+k)|\\
&\le2|I\symd\mathcal J|+2k.
\end{aligned}
\]
Consequently,
\begin{equation}\label{eq:oddhalf-shift}
|I\symd(rJ+d)|\ge\dist(kr,N\Z)-k,
\quad k\in\N.
\end{equation}
For odd $N$, suitable choices of $k$ give the following bound.

\begin{lemma}\label{lem:oddhalf}\label{lem:oddhalf-third}
Let $N\ge5$ be odd, and let $I,J\subset\Z/N\Z$ be intervals,
each of length $(N-1)/2$ or $(N+1)/2$. If $r$ is a unit modulo $N$
with $r\not\equiv\pm1$, then, for every $d\in\Z/N\Z$,
\[
\left\lfloor\frac N3\right\rfloor
\le |I\symd(rJ+d)|
\le N-\left\lfloor\frac N3\right\rfloor.
\]
\end{lemma}

\begin{proof}
Let $\mathcal J=rJ+d$. The complement of $\mathcal J$ is again an
affine image of an interval of one of the two specified lengths, so
it suffices to prove the lower bound.

Let $q=\lfloor N/3\rfloor$. Changing the sign of $r$ and reversing $J$,
we may assume $2\le r\le(N-1)/2$. Inversion is also allowed:
multiplication by $r^{-1}$ and exchange of $I,J$ preserve the
symmetric-difference size, after a translation.
For $r=2$, replacing $r$ by $-r^{-1}$ gives $r=(N-1)/2$.

Suppose $N\ge31$, so $q\ge10$.
If $r\ge q+1$, take $k=1$ in \eqref{eq:oddhalf-shift}.
If $r=q$, coprimality excludes $N=3q$.
For $N=3q+1$, replacing $r$ by $-r^{-1}$ gives $r=3$;
for $N=3q+2$, replacing $r$ by $r^{-1}$ gives $r=(N-3)/2\ge q+1$.
For $r=3$, coprimality leaves
\[
(N,k)=(3q+1,q/2)\quad\text{or}\quad
(N,k)=(3q+2,(q+1)/2).
\]
Here $q$ is even in the first case and odd in the second, and
\[
k(r-1)\ge q,\qquad k(r+1)\le N-q.
\]

For $4\le r\le q-1$, let $k=\lceil q/(r-1)\rceil$.
Then $k(r-1)\ge q$. If $r-1\ge q/2$, we have $k=2$ and
$k(r+1)\le2q$. Otherwise $4\le r\le(q+1)/2$, and
\[
k(r+1)\le q+r+\frac{2(q-1)}{r-1}\le2q.
\]
The last inequality follows by convexity in $r$: the endpoint values
are $(5q+10)/3$ and $(3q+9)/2$, both at most $2q$ for $q\ge10$.
Thus in each remaining case
\[
q+k\le kr\le N-q-k,
\]
and \eqref{eq:oddhalf-shift} gives $|I\symd\mathcal J|\ge q$.

For $N<31$, the following representatives cover the nontrivial units
under sign and inversion. Each listed $k$ satisfies
$\dist(kr,N\Z)-k\ge\lfloor N/3\rfloor$.
\[
\renewcommand{\arraystretch}{1.25}
\begin{array}{c@{\qquad}l@{\qquad}l}
\toprule
N&r&k\\
\midrule
5&2&1\\
7&3&1\\
9&4&1\\
11&5,4&1,1\\
13&6,3,5&1,2,1\\
15&7,4&1,2\\
17&8,6,4,7&1,1,2,1\\
19&9,3,4,7&1,3,2,1\\
21&10,5,8&1,2,1\\
23&11,8,6,9,10&1,1,2,1,1\\
25&12,3,6,7,9&1,4,2,2,1\\
27&13,7,11,10&1,2,1,1\\
29&14,10,7,6,11,13,12&1,1,2,2,1,1,1\\
\bottomrule
\end{array}
\]
This proves the lower bound for every odd $N\ge5$.
\end{proof}

\subsection{Proof of the conditional correlation bound}\label{sec:conditional-proof}

\begin{proof}[Proof of Proposition~\ref{lem:sampled-half}]
Replacing $j$ by $-j$ gives $F_a(-\tau)=F_a(\tau)$, and hence
\[
\rho=\frac3\pi\int_0^{\pi/3}F_a(\tau)\,d\tau.
\]
For $a=1$, Lemma~\ref{lem:integerprefix} gives the bound, with equality
exactly when $\{b,c\}=\{2,6\}$. Assume henceforth that $a\ge3$.
We prove strict inequality when $c\equiv\pm b\pmod a$; in the other
case, only the five-point grid can give equality.

\emph{Congruent frequencies: $c\equiv\pm b\pmod a$.}
First suppose $a=3$. Let $\epsilon\in\{\pm1\}$ and $k\in\Z$ satisfy
$\epsilon c=b+3k$. Since $b\ne c$, we have $k\ne0$.
The sampled products are $s(z_j)s(z_j+k\tau)$, with
$z_j=b(\tau+2\pi j)/3$. These points are equally spaced modulo $\pi$.
In a $\pi$-period, the negative set of $s(z)s(z+w)$ is an interval of
length $\dist(w,2\pi\Z)$. Counting its grid points gives
\[
F_3(\tau)\ge-\frac13+\frac23H(k\tau),\qquad
H(t)=\one\{\cos t>1/2\}-\one\{\cos t<-1/2\}.
\]
For $n=|k|$, integration over full periods and the remainder gives
\[
\frac3\pi\int_0^{\pi/3}H(k\tau)\,d\tau
=
\begin{cases}
\dfrac1n,&n\equiv1,2\pmod6,\\[6pt]
0,&n\equiv0,3\pmod6,\\[6pt]
-\dfrac1n,&n\equiv4,5\pmod6.
\end{cases}
\]
This is at least $-1/4$, so $\rho\ge-1/2$.
The inequality is strict: near $\tau=0$, $F_3(\tau)=1$, while its
pointwise lower bound is $1/3$.

Now suppose $a\ge5$. Choose $\epsilon\in\{\pm1\}$ and a nonzero
integer $k$ with $\epsilon c=b+ka$. By \eqref{eq:congruent-grid},
\[
F_a(\tau)\ge\Lambda(k\tau)-\frac2a.
\]
For $n=|k|$,
\[
\frac3\pi\int_0^{\pi/3}\Lambda(k\tau)\,d\tau
=
\begin{cases}
\dfrac2{3n},&n\equiv1,2\pmod6,\\[6pt]
0,&n\equiv0,3\pmod6,\\[6pt]
-\dfrac2{3n},&n\equiv4,5\pmod6.
\end{cases}
\]
For $a\ge7$, this gives
\[
\rho\ge-\frac16-\frac27=-\frac{19}{42}>-\frac12.
\]
For $a=5$ and $n\notin\{4,5\}$, it gives
\[
\rho\ge-\frac1{15}-\frac25=-\frac7{15}>-\frac12.
\]
For $n=4,5$, the ceiling bound in \eqref{eq:congruent-grid} gives
\[
F_5(\tau)\ge1-\frac25
\left\lceil\frac5\pi\dist(n\tau,2\pi\Z)\right\rceil.
\]
The two ceiling averages are
\[
\begin{aligned}
\frac34\left(\int_0^1\lceil5u\rceil\,du+
\int_{2/3}^1\lceil5u\rceil\,du\right)&=\frac{17}{5},\\[6pt]
\frac35\left(\int_0^1\lceil5u\rceil\,du+
\int_{1/3}^1\lceil5u\rceil\,du\right)&=\frac{83}{25}.
\end{aligned}
\]
Hence $\rho\ge-9/25$ for $n=4$, and $\rho\ge-41/125$ for $n=5$.
Both bounds exceed $-1/2$, completing the congruent case.

\emph{Noncongruent frequencies: $c\not\equiv\pm b\pmod a$.}
Here $a\ge5$. If $a\ge7$, the two positive sets on the sampled grid are affine images of intervals of lengths $(a\pm1)/2$.
Their relative multiplier is not $\pm1$. Lemma~\ref{lem:oddhalf} gives
\[
F_a(\tau)\ge-1+\frac2a\left\lfloor\frac a3\right\rfloor>-\frac12,
\]
since $\lfloor a/3\rfloor>a/4$ for odd $a\ge7$.

It remains to treat $a=5$.
Choose $\epsilon\in\{\pm1\}$ and $k\in\Z$ with
\[
\epsilon c=2b+5k.
\]
After reindexing, the products are $s(z_j)s(2z_j+\delta)$, where
\[
z_j=z+\frac{2\pi j}{5},\qquad z=\frac{b\tau}{5},\qquad
\delta=k\tau.
\]
Reducing $z$ modulo $2\pi/5$ amounts to another reindexing. Choose
$z\in(-\pi/10,3\pi/10)$, and let $\zeta=2z+\delta\pmod{2\pi}$.
The first positive set is $\{0,1,4\}$ for $|z|<\pi/10$ and
$\{0,4\}$ for $\pi/10<z<3\pi/10$.
The second positive set consists of consecutive indices in the order
$(0,3,1,4,2)$. Neither complement of the first set can occur.
Thus $F_5\ge-3/5$, with equality exactly in the following two cases:
\[
\renewcommand{\arraystretch}{1.3}
\begin{array}{c@{\qquad}c}
\toprule
z&\zeta\pmod{2\pi}\\
\midrule
(-\pi/10,\pi/10)&(-\pi/10,\pi/10)\\
(\pi/10,3\pi/10)&(13\pi/10,15\pi/10)\\
\bottomrule
\end{array}
\]
Otherwise $F_5\ge-1/5$. Since $\delta=\zeta-2z\pmod{2\pi}$,
\[
F_5(\tau)=-\frac35
\quad\Longrightarrow\quad
\dist(k\tau,\pi\Z)<\frac{3\pi}{10}.
\]
For $n=|k|>0$, the proportion of $\tau\in[0,\pi/3]$ satisfying the
last condition is
\[
\begin{cases}
\dfrac35,&n\equiv0\pmod3,\\[6pt]
\dfrac35+\dfrac3{10n},&n\equiv1\pmod3,\\[6pt]
\dfrac35-\dfrac3{10n},&n\equiv2\pmod3.
\end{cases}
\]
For $n\ge2$ this is at most $27/40$, whence
\[
\rho\ge-\frac15-\frac25\cdot\frac{27}{40}
=-\frac{47}{100}>-\frac12.
\]
If $|k|=1$, interchange $b,c$ and replace $\epsilon$ by $-\epsilon$.
The new integer is
\[
k'=-\epsilon(b+2k).
\]
For $k=1$, $|k'|\ge3$. For $k=-1$, the condition $|k'|<2$ forces
$b\in\{1,2,3\}$, giving only the unordered pairs $\{1,2\}$ and
$\{1,3\}$. The latter has
\[
F_5(\tau)=
\begin{cases}
-\dfrac35,&0<\tau<\pi/6,\\[6pt]
-\dfrac15,&\pi/6<\tau<\pi/3,
\end{cases}
\qquad \rho=-\frac25.
\]
The pair $\{1,2\}$ belongs to the case $k=0$, after ordering.

If $k=0$, positivity forces $\epsilon=1$ and $c=2b$.
Counting the five-point grid gives
\[
F_5(\tau)=-\frac15\bigl(2s(b\tau)+s(b\tau)s(2b\tau)\bigr).
\]
For $b=1$, the averages of the two terms $s(\tau)$ and
$s(\tau)s(2\tau)$ are $1$ and $1/2$, so $\rho=-1/2$.
For $b\ge2$, the primitive of $s$ gives
\[
\frac3\pi\int_0^{\pi/3}s(b\tau)\,d\tau
\le
\begin{cases}
\dfrac12,&b=2,\\[6pt]
\dfrac3{2b}\le\dfrac12,&b\ge3.
\end{cases}
\]
Since $s(b\tau)s(2b\tau)\le1$, we obtain $\rho\ge-2/5$.

Thus the only equality case with $a\ge3$ is
$(a,\{b,c\})=(5,\{1,2\})$, as required.
\end{proof}

\subsection{Even grids and quadrant kernels}

For the even $3$-pair and $7$-pair configurations of
Section~\ref{sec:mixed}, we need the even analogue of
Lemma~\ref{lem:oddhalf}. The same translation estimate
\eqref{eq:oddhalf-shift} gives the required bound.

\begin{lemma}\label{lem:evenhalf}
Let $c\in\N$ with $c\ge2$, and put $N=4c$ and $q=2\lfloor N/6\rfloor$.
Let $I\subset\Z/N\Z$ be an interval of $2c$ points.
If $r$ is a unit modulo $N$ with $r\not\equiv\pm1$, then, for every
$d\in\Z/N\Z$,
\begin{equation}\label{eq:evenhalf-third}
q\le|I\symd(rI+d)|\le N-q.
\end{equation}
In particular,
\[
c\le|I\symd(rI+d)|\le3c,
\]
and, if $c\ge5$,
\[
c+1\le|I\symd(rI+d)|\le3c-1.
\]
\end{lemma}

\begin{proof}
Let $\mathcal J=rI+d$. By \eqref{eq:oddhalf-shift},
\begin{equation}\label{eq:evenhalf-shift}
|I\symd\mathcal J|\ge\dist(kr,N\Z)-k,\quad k\in\N.
\end{equation}
Its complement is another translate of $rI$, so it suffices to prove
 the lower bound. Reversing $r$ and translating,
we may assume $3\le r\le N/2-1$; every unit modulo $N$ is odd.

The integer $q$ is even and $3q\le N$.
If $r\ge q+1$, take $k=1$; if $r=3$, take $k=q/2$.
Both choices give $\dist(kr,N\Z)-k\ge q$.
For $5\le r\le q-1$ and $q\ge10$, let
$k=\lceil q/(r-1)\rceil$. Then $k(r-1)\ge q$.
If $r-1\ge q/2$, then $k\le2$ and $k(r+1)\le2q$.
Otherwise $r\le q/2$, and
\[
k(r+1)\le q+r+\frac{2(q-1)}{r-1}\le2q.
\]
For the last inequality, convexity on $[4,q/2]$ reduces the claim to
\[
\frac{5q+10}{3}\le2q,\qquad
\frac{3q}{2}+4+\frac4{q-2}\le2q,
\]
which hold for $q\ge10$. Hence
\[
q+k\le kr\le N-q-k,
\]
and \eqref{eq:evenhalf-shift} proves the lower bound.

For $N<32$, the preceding cases leave only
\[
(N,r)=(24,5),\ (24,7),\ (28,5).
\]
Taking $k=2$ proves the bound in each case.
Finally, $q\ge c$ for $c\ge2$, and $q\ge c+1$ for $c\ge5$.
\end{proof}

For the smallest even modulus, the sampled correlation reduces to
one of the quadrant kernels
\[
K_\pm(t)=\frac12\bigl(\sgn\cos\alpha t\,\sgn\cos\beta t
\pm\sgn\sin\alpha t\,\sgn\sin\beta t\bigr),
\qquad \alpha,\beta>0.
\]
Their averages on initial intervals satisfy the same lower bound.

\begin{lemma}\label{lem:quadrant}
For every $T>0$,

\[
\int_0^TK_\pm(t)\,dt\ge-\frac T2.
\]
\end{lemma}

\begin{proof}
If $n\pi/2<\theta<(n+1)\pi/2$, then
\[
(\sgn\cos\theta,\sgn\sin\theta)
=\sqrt2\bigl(\cos\frac{(2n+1)\pi}4,\sin\frac{(2n+1)\pi}4\bigr).
\]
Thus, if $\alpha t$ and $\beta t$ have quadrant indices $m$ and $n$,
\begin{equation}\label{eq:Kpm}
K_-(t)=-\sin\left(\frac{\pi(n+m)}2\right),\qquad
K_+(t)=\cos\left(\frac{\pi(n-m)}2\right).
\end{equation}
We retain the notation $K_\pm$ for the kernels after rescaling the
independent variable.

\emph{The kernel $K_-$.} If $\alpha=\beta$, then $K_-\equiv0$.
Otherwise, by symmetry, assume $\alpha<\beta$. Let
$w=2\beta t/\pi$, $W=2\beta T/\pi$, and
$\lambda=\alpha/\beta\in(0,1)$, and write
$D(w)=\lfloor w\rfloor+\lfloor\lambda w\rfloor$.
By \eqref{eq:Kpm}, the negative set is
$E=\{w\in[0,W]:D(w)\equiv1\pmod4\}$.
It lies in $[1,W]$, and for $w\in E$,
\[
D(w)-D(w-1)
=1+\lfloor\lambda w\rfloor-\lfloor\lambda(w-1)\rfloor\in\{1,2\}.
\]
Thus $E\cap(E-1)=\varnothing$ and $E\cup(E-1)\subset[0,W]$, whence
\[
2|E|\le W,\qquad \int_0^W K_-(w)\,dw\ge-|E|\ge-\frac W2.
\]

\emph{The kernel $K_+$.} If $\alpha=\beta$, then $K_+\equiv1$.
Otherwise assume $\alpha<\beta$, and let
\[
u=\frac{2(\beta-\alpha)t}{\pi},\qquad
U=\frac{2(\beta-\alpha)T}{\pi},\qquad
\gamma=\frac{\alpha}{\beta-\alpha}.
\]
Now $D(u)=\lfloor(\gamma+1)u\rfloor-\lfloor\gamma u\rfloor=n-m$.
Since $|D(u)-u|<1$, the condition $K_+=-1$, equivalently
$D\equiv2\pmod4$, can hold only on
\[
\mathcal N=\bigcup_{j\ge0}(4j+1,4j+3),\qquad
|\mathcal N\cap[0,U]|\le\frac U2+\frac12.
\]
On $[0,1]$, we have $D\in\{0,1\}$, so $K_+\ge0$ there.
Let $q=\lfloor\gamma\rfloor$ and $\theta=\gamma-q$.
On $[0,1]$, $K_+=1$ precisely where $D=0$, namely on
\[
\bigcup_{j=0}^{q}\left[\frac j\gamma,\frac{j+1}{\gamma+1}\right).
\]
This set has measure
\[
\begin{aligned}
M=\frac{(q+1)(2\gamma-q)}{2\gamma(\gamma+1)}=\frac12+\frac{\theta(1-\theta)}{2\gamma(\gamma+1)}\ge\frac12.
\end{aligned}
\]
Therefore
\[
\int_0^UK_+(u)\,du\ge
\begin{cases}
0,&U\le1,\\[6pt]
\dfrac12-|\mathcal N\cap[0,U]|\ge-\dfrac U2,&U\ge1.
\end{cases}
\]
Rescaling proves the assertion.
\end{proof}

\section{Mixed parity}\label{sec:mixed}

By Corollary~\ref{cor:oneodd}, Theorem~\ref{thm:mixed} reduces to
two odd and two even frequencies. We distinguish whether neither
same-parity pair has ratio three, the odd pair has ratio three, or the
even pair has ratio three. The first case gives strict inequality;
the equality configurations arise in the latter two cases.

\subsection{No pair in ratio three}\label{sec:2o2e}

\begin{proposition}\label{prop:2o2e}
Let $a,b$ be distinct odd positive integers and $c,d$ distinct even
positive integers. If neither pair is a $3$-pair, then
\[
P(a,b,c,d)>\frac1{24}.
\]
\end{proposition}

We first derive a parity-averaging bound for the fourfold correlation.
This reduces the proof to the two $7$-pair configurations treated below.

Let $a,b$ be distinct odd positive integers and $c,d$ distinct even
positive integers. Write $P=P(a,b,c,d)$ and
\[
\begin{gathered}
F(x)=s(ax)s(bx),\qquad G(x)=s(cx)s(dx),\\
\mathcal A=C(a,b),\qquad \mathcal B=C(c,d),\qquad \eta=(-1)^{(c+d)/2}.
\end{gathered}
\]
The odd--even correlations vanish by Lemma~\ref{lem:pair}.
Lemma~\ref{lem:expansion} therefore gives
\begin{equation}\label{eq:2o2e}
8P=1+\mathcal A+\mathcal B+\E[FG].
\end{equation}

\begin{lemma}\label{lem:FG}
We have
\[
\E[FG]\ge-\frac{1+\eta\mathcal A}{2}.
\]
\end{lemma}

\begin{proof}
Since $G(x+\pi/2)=\eta G(x)$, averaging over the shift $x\mapsto x+\pi/2$ gives
\[
\begin{aligned}
\E[FG]
&=\E\left[G(x)\frac{F(x)+\eta F(x+\pi/2)}2\right]\\
&\ge-\E\left|\frac{F(x)+\eta F(x+\pi/2)}2\right|
=-\frac{1+\eta\mathcal A}2.
\end{aligned}
\]
Here $F=\pm1$, and \eqref{eq:doubling} gives
\[
\E[F(x)F(x+\pi/2)]=\E[F(2x+\pi/2)]=\mathcal A.
\]
\end{proof}

When neither pair has ratio three, Lemma~\ref{lem:pair} gives
$\mathcal A,\mathcal B\ge-1/7$. The next possible negative value is
$-1/11$. Combining these thresholds with Lemma~\ref{lem:FG}, the proof
below reduces to
\[
\eta=-1,\quad\mathcal A=-\frac17
\qquad\text{or}\qquad
\eta=1,\quad\mathcal B=-\frac17.
\]
These are the odd and even $7$-pair configurations, respectively.

\Needspace{9\baselineskip}
The first estimate treats an odd $7$-pair when the even frequencies
have different $2$-adic valuations.

\begin{lemma}\label{lem:oddseven}
Let $a,b$ be odd positive integers and $c$ an even positive integer,
with $\gcd(a,bc)=1$. Then
\[
P(a,7a,2b,2c)>\frac1{24}.
\]
\end{lemma}

\begin{proof}
Let $H=E_7-E_7(\cdot+\pi)$. Since $s(by)s(cy)$ changes sign under
$y\mapsto y+\pi$, the substitution $y=2x$ gives
\[
8P=\frac67+Q,\qquad Q=\E[H(ay)s(by)s(cy)].
\]
It suffices to prove $Q>-11/21$. Averaging over the common factor of
$b,c$ will leave only the factors $1$ and $3$. For these, we use the
initial-interval and grid estimates of Section~\ref{sec:auxiliary}.
Let $g=\gcd(b,c)$, $b_0=b/g$, and $c_0=c/g$.
Then $g$ is odd, $\gcd(a,g)=1$, and $a,b_0,c_0$ are pairwise coprime.
Conditioning on $gy$ gives
\[
Q=\E[(T_gH)(ay)s(b_0y)s(c_0y)].
\]
The numbers of points of a $g$-point grid in two arcs of the same
length differ by at most one. Since $E_7$ has three components,
\[
\|T_gH\|_\infty\le\frac3g.
\]
Thus $g\ge7$ gives $Q\ge-3/7>-11/21$.
For $g=3,5$, averaging on the arcs gives
\[
\renewcommand{\arraystretch}{1.2}
\begin{array}{c*{7}{r}}
\toprule
j&0&1&2&3&4&5&6\\
\midrule
3T_3H&3&2&1&0&-1&-2&-3\\
5T_5H&-1&0&-1&0&1&0&1\\
\bottomrule
\end{array}
\qquad\left(\frac{j\pi}{7}<y<\frac{(j+1)\pi}{7}\right).
\]
Both functions are even and change sign under translation by $\pi$.
In particular, $g=5$ gives $Q\ge-1/5$.

It remains to consider $g=1,3$. Let $F_a$ be the sampled
correlation \eqref{eq:sampled-correlation} with $(b,c)=(b_0,c_0)$.
Then
\[
F_a(-\tau)=F_a(\tau),\qquad F_a(\tau+\pi)=-F_a(\tau).
\]
The positive set of $H$ is
\[
(-\pi/7,\pi/7)\cup(4\pi/7,5\pi/7)\cup(-5\pi/7,-4\pi/7).
\]
Conditioning on $ay$, and using the table for $T_3H$, gives
\[
Q=\frac47\rho_g,
\]
where
\[
\begin{aligned}
\rho_1&=\frac7{2\pi}\left(
\int_0^{\pi/7}F_a(\tau)\,d\tau+
\int_{4\pi/7}^{5\pi/7}F_a(\tau)\,d\tau\right),\\
\rho_3&=\frac7{6\pi}\sum_{j=1}^3
\int_0^{j\pi/7}F_a(\tau)\,d\tau.
\end{aligned}
\]
For $a=1$, Lemma~\ref{thm:prefix} on the initial intervals and
$F_1\ge-1$ on the other interval give
\[
\rho_1\ge-\frac45,\qquad \rho_3\ge-\frac35.
\]
For $a\ge5$ and $c_0\not\equiv\pm b_0\pmod a$,
Lemma~\ref{lem:oddhalf} gives $F_a\ge-3/5$, hence $\rho_g\ge-3/5$.

The remaining case has $a\ge3$ and
$\epsilon c_0=b_0+ka$ for some $\epsilon\in\{\pm1\}$ and a nonzero
odd integer $k$. Applying \eqref{eq:congruent-grid} to $(b_0,c_0)$ gives
\[
F_a(\tau)\ge\Lambda(k\tau)-\frac2a.
\]
Let $D_g(k)$ be the average obtained by replacing $F_a(\tau)$ with
$\Lambda(k\tau)$ in $\rho_g$. For $k>0$,
\[
\begin{aligned}
D_1(k)&=\frac7{2k}\bigl[J(k/7)+J(5k/7)-J(4k/7)\bigr],\\
D_3(k)&=\frac7{6k}\sum_{j=1}^3J(jk/7),
\end{aligned}
\]
where $J$ is $2$-periodic and
\[
J(t)=
\begin{cases}
t-t^2,&0\le t\le1,\\[3pt]
t^2-3t+2,&1\le t\le2.
\end{cases}
\]
The numerators have period $14$:
\[
\renewcommand{\arraystretch}{1.2}
\begin{array}{c*{7}{c}}
\toprule
k\bmod14&1&3&5&7&9&11&13\\
\midrule
kD_1(k)&2/7&2&-4/7&0&4/7&-2&-2/7\\
kD_3(k)&2/3&4/21&2/21&0&-2/21&-4/21&-2/3\\
\bottomrule
\end{array}
\]
Since $D_g(-k)=D_g(k)$,
\[
D_1(k)\ge-\frac2{11},\qquad D_3(k)\ge-\frac2{39}.
\]
Thus
\[
\rho_1\ge-\frac2{11}-\frac23=-\frac{28}{33},
\qquad
\rho_3\ge-\frac2{39}-\frac23=-\frac{28}{39}.
\]
For $g=1,3$, these bounds give $\rho_g\ge-28/33$, hence
\[
Q\ge-\frac{16}{33}>-\frac{11}{21}.\qedhere
\]
\end{proof}

The second estimate treats an even $7$-pair. Here the cyclic
half-circle bound applies unless the odd frequencies are congruent
up to sign modulo twice the smaller even frequency.

\begin{lemma}\label{lem:seven-all}
Let $a<b$ be odd positive integers and $h$ an even positive integer,
with $\gcd(h,ab)=1$. If $b\ne3a$, then
\[
P(a,b,h,7h)>\frac1{24}.
\]
\end{lemma}

\begin{proof}
On $\R/\Z$, let
\[
\begin{gathered}
W(z)=E_7(2\pi z),\qquad
E(t)=\frac{1+s(\pi at)s(\pi bt)}2,\qquad
U_h(z)=\frac1h\sum_{j=0}^{h-1}E\left(\frac{z+j}{h}\right).
\end{gathered}
\]
The components of $W$ are
\[
(-1/14,1/14),\qquad(3/14,5/14),\qquad(9/14,11/14).
\]
Thus $W$ has mean $3/7$. Equation~\eqref{eq:parity-split}, followed by
conditioning on the even pair, gives
\begin{equation}\label{eq:seven-all-average}
P(a,b,h,7h)=\frac12\int_0^1W(z)U_h(z)\,dz.
\end{equation}
When $a\not\equiv\pm b\pmod{2h}$, the half-circle estimate applies.
In the congruent case, the quadrant identities below reduce the proof
to $k=2$, and the grid estimate then leaves only $h=2,4$.

Suppose first that $a\not\equiv\pm b\pmod{2h}$.
Then $h\ge4$, and Lemma~\ref{lem:evenhalf} applies to the $2h$
sample points. Each positive half-circle contains $h$ points, and
their symmetric difference has at most $3h/2$ points. Thus
\[
U_h\ge\frac14,
\qquad P\ge\frac3{56}>\frac1{24}.
\]
Otherwise, let
\[
a\equiv\epsilon b\pmod{2h},\qquad
\epsilon\in\{\pm1\},\qquad
k=\frac{|a-\epsilon b|}{2h}\ge1.
\]
Since $4\mid2h$, we have $\epsilon=\chi(ab)$.
Pairing the indices $j$ and $j+h/2$ expresses $U_h$ as an average
of $(1+K_\epsilon(\alpha,\beta))/2$, where
\[
K_\epsilon(\alpha,\beta)
=\frac12\bigl(s(\alpha)s(\beta)
+\epsilon\,\sgn\sin\alpha\,\sgn\sin\beta\bigr).
\]
The product-to-sum identities give
\[
K_\epsilon(\alpha,\beta)
=\sgn\cos(\alpha-\epsilon\beta)\,
\one_{\{|\cos(\alpha-\epsilon\beta)|
>|\cos(\alpha+\epsilon\beta)|\}}.
\]
For the sampled angles,
$\cos(\alpha-\epsilon\beta)=\cos(2\pi kz)$, independently of $j$.
Consequently,
\[
U_h(z)\ge\frac12\one_{\{\cos(2\pi kz)>0\}},
\qquad
P\ge\frac14C(k),
\]
where
\[
C(k)=\int_0^1W(z)\one_{\{\cos(2\pi kz)>0\}}\,dz.
\]
Let $A$ be the $1$-periodic primitive of $s(2\pi t)$ given by
\[
A(t)=
\begin{cases}
t,&0\le t\le1/4,\\[3pt]
1/2-t,&1/4\le t\le3/4,\\[3pt]
t-1,&3/4\le t\le1.
\end{cases}
\]
Integration over the three arcs yields
\[
C(k)=\frac3{14}
+\frac{A(k/14)+A(5k/14)-A(3k/14)}{k}.
\]
The numerator in the following table is $14$-periodic:
\[
\setlength{\arraycolsep}{3.5pt}
\renewcommand{\arraystretch}{1.2}
\begin{array}{c*{14}{r}}
\toprule
k\bmod14&1&2&3&4&5&6&7&8&9&10&11&12&13&0\\
\midrule
7k(C(k)-3/14)&0&-1&3&3&-1&0&0&0&1&-3&-3&1&0&0\\
\bottomrule
\end{array}
\]
For a negative numerator $-1$, $C(k)>1/6$ whenever $k>3$;
for $-3$, it suffices that $k>9$.
Hence $P>1/24$ unless $k=2$.

Assume $k=2$. Since $\gcd(a,h)=1$, the sampled products
$s(u)s(u+4\pi z)$ form a uniform $h$-point grid modulo $\pi$.
Lemma~\ref{lem:grid-correlation} therefore gives
\[
U_h(z)\ge\frac{1+\Lambda(4\pi z)}2-\frac1h.
\]
Integration on the three arcs gives
\[
\int_0^1W(z)\Lambda(4\pi z)\,dz=-\frac4{49},
\qquad
P\ge\frac{17}{196}-\frac3{14h}.
\]
For $h\ge6$, this is at least $5/98>1/24$.

Only $h=2,4$ remain. If $\epsilon=-1$, then $a+b=4h$, leaving
respectively two and four positive odd pairs.
If $\epsilon=1$, then $b=a+4h$. Let
\[
q(z)=(-1)^{\lfloor4z\rfloor},\qquad
R_n(z)=\{nz\}-\frac12.
\]
Let $m=\lfloor4z\rfloor$ and $r=4z-m$. Since $h$ is even, the
negative interval of $s(u)s(u+\pi r)$ contains exactly
\[
\lfloor az+hr\rfloor-\lfloor az\rfloor
=\lfloor bz\rfloor-\lfloor az\rfloor-hm
\]
sample points. Thus
\[
U_h(z)=\frac{1+\Lambda(4\pi z)}2
+\frac{q(z)}h\bigl(R_b(z)-R_a(z)\bigr),
\]
and
\[
P=\frac{17}{196}+\frac1{2h}\int_0^1W(z)q(z)
\bigl(R_b(z)-R_a(z)\bigr)\,dz.
\]
The function $Wq$ has six nonzero constant-sign intervals.
A periodic primitive of $R_n$ is $B_2(\{nz\})/(2n)$, whose range is
$1/(8n)$. Hence
\[
P\ge\frac{17}{196}-\frac6{16h}\left(\frac1a+\frac1b\right)
>\frac{17}{196}-\frac3{4ha}.
\]
This exceeds $1/24$ when
\[
a>\frac{882}{53h}.
\]
For $h=2,4$, the thresholds are $441/53<9$ and $441/106<5$.
The remaining cases have the following values:
\[
\renewcommand{\arraystretch}{1.6}
\begin{array}{cc@{\qquad}cc}
\toprule
(a,b,h)&P&(a,b,h)&P\\
\midrule
(1,7,2)&\dfrac1{14}&(1,15,4)&\dfrac{17}{210}\\[5pt]
(3,5,2)&\dfrac1{21}&(3,13,4)&\dfrac{83}{1092}\\[5pt]
(1,9,2)&\dfrac{11}{126}&(5,11,4)&\dfrac{23}{308}\\[5pt]
(3,11,2)&\dfrac{47}{462}&(7,9,4)&\dfrac5{63}\\[5pt]
(5,13,2)&\dfrac{19}{182}&(1,17,4)&\dfrac{43}{476}\\[5pt]
(7,15,2)&\dfrac{19}{210}&(3,19,4)&\dfrac{151}{1596}\\
\bottomrule
\end{array}
\]
These values follow by partitioning $[0,1]$ at the zeros
$(2j+1)/(4n)$, with $n\in\{a,b,h,7h\}$, and summing the common-sign
intervals. Every value is greater than $1/24$.
\end{proof}

\begin{proof}[Proof of Proposition~\ref{prop:2o2e}]
Divide by the common divisor, which is odd.
If three frequencies have a common factor greater than $1$,
\eqref{eq:mixed-transfer} gives $P\ge1/18>1/24$.
We may therefore assume \eqref{eq:tripgcd}.
In the notation preceding Lemma~\ref{lem:FG},
$\mathcal A,\mathcal B\ge-1/7$.

If $\eta=-1$, then $\mathcal B=0$, and Lemma~\ref{lem:FG} gives
\[
P\ge\frac{1+3\mathcal A}{16}.
\]
For $\mathcal A\ge-1/11$, this is at least $1/22>1/24$.
Otherwise $\mathcal A=-1/7$. After interchanging frequencies within
the pairs, the configuration is $(r,7r,2u,2w)$, with $r,u$ odd and
$w$ even. Condition~\eqref{eq:tripgcd} gives $\gcd(r,uw)=1$, so
Lemma~\ref{lem:oddseven} applies.

If $\eta=1$, Lemma~\ref{lem:FG} gives
\[
P\ge\frac{1+\mathcal A+2\mathcal B}{16}.
\]
For $\mathcal B\ge-1/11$, this is at least
\[
\frac{1-1/7-2/11}{16}=\frac{13}{308}>\frac1{24}.
\]
Otherwise $\mathcal B=-1/7$, so the even pair is $\{h,7h\}$.
Here $h$ is even and $\gcd(h,ab)=1$ by \eqref{eq:tripgcd}.
Lemma~\ref{lem:seven-all} completes the proof.
\end{proof}

\subsection{The odd pair has ratio three}\label{sec:oddtriple}

For an odd $3$-pair, the common-sign condition becomes the central-arc
condition used in Proposition~\ref{lem:sampled-half}. Its two equality cases
give exactly the two extremal families.

\begin{proposition}\label{thm:oddtriple}
Let $a$ be an odd positive integer and let $c_0,d_0$ be distinct even
positive integers.
Then
\[
P(a,3a,c_0,d_0)\ge\frac1{24},
\]
with equality if and only if, for some odd $d\in\N$,
\[
\{a,3a,c_0,d_0\}=d\{1,3,4,12\}
\quad\text{or}\quad
\{a,3a,c_0,d_0\}=d\{2,4,5,15\}.
\]
\end{proposition}

\begin{proof}
Divide by the common divisor, which is odd.
If three frequencies have a common factor greater than $1$,
\eqref{eq:mixed-transfer} gives $P\ge1/18$.
Otherwise let $(c_0,d_0)=(2b,2c)$.
Then $\gcd(a,bc)=1$ by \eqref{eq:tripgcd}.
By \eqref{eq:oddtriple-conditioning} and
Proposition~\ref{lem:sampled-half},
\[
P(a,3a,2b,2c)=\frac{1+\rho}{12}\ge\frac1{24},
\]
with equality precisely when
\[
(a,\{b,c\})=(1,\{2,6\})\quad\text{or}\quad(5,\{1,2\}).
\]
Restoring the common divisor gives the two stated families.
\end{proof}

\subsection{The even pair has ratio three}\label{sec:eventriple}

It remains to condition on an even $3$-pair. Common factors are handled
first; the cyclic half-circle bounds then leave only the congruences
$a\equiv\pm b\pmod{4c}$ and the smallest modulus.

\begin{proposition}\label{thm:eventriple}
Let $a<b$ be odd positive integers and let $c\in\N$. Then
\[
P(a,b,2c,6c)\ge\frac1{24},
\]
with equality if and only if
\[
(a,b,c)=d(1,3,2)\qquad\text{for some odd }d\in\N.
\]
\end{proposition}

\begin{proof}
By dilation, assume $\gcd(a,b,c)=1$. If $b=3a$, the assertion follows
from Proposition~\ref{prop:doubletriple}(b), with equality precisely
when $(a,b,c)=(1,3,2)$. We may therefore assume $b\ne3a$ and prove
strict inequality. In particular, $C(a,b)\ge-1/7$.

By \eqref{eq:s3}, the pair $(2c,6c)$ has a common sign precisely on
$\mathcal C_c=\{x:\cos(4cx)>1/2\}$, a set of normalized measure
$1/3$. With $F=s(ax)s(bx)$, Equation~\eqref{eq:parity-split} gives
\begin{equation}\label{eq:E1}
\begin{gathered}
P(a,b,2c,6c)=\frac14\E[(1+F)\one_{\mathcal C_c}]
=\frac{1+\rho}{12},\\
\rho=\E[F\mid\mathcal C_c].
\end{gathered}
\end{equation}
We prove $\rho>-1/2$. We first remove common factors, leaving
$a,b,c$ pairwise coprime. The remaining cases are
\[
a\not\equiv\pm b\pmod{4c},\qquad c=1,\qquad
c\ge2\ \text{and}\ a\equiv\pm b\pmod{4c}.
\]
The first is controlled by discrete half-circles, the second by
quadrant kernels, and the third by discrete autocorrelation. Each
case will give strict inequality, so no further equality cases arise.

\emph{Common factors.} If $h=\gcd(a,c)>1$, then $h$ is odd and
$\gcd(h,b)=1$. For $T=(a/h,2c/h,6c/h)$,
\[
P(T)=\frac{1+C(2,6)}4=\frac16.
\]
Proposition~\ref{prop:transfer} gives
\[
P\ge\frac{1-1/h}{2}\cdot\frac16\ge\frac1{18}.
\]
The same argument applies when $\gcd(b,c)>1$.

Now let $g=\gcd(a,b)>1$, and write $a=gu$, $b=gv$.
Then $\gcd(g,c)=1$, and $\mathcal A=C(u,v)\ge-1/7$.
Let $E(t)=\one\{\dist(t,\Z)<1/6\}$ on $\R/\Z$.
Thus $\one_{\mathcal C_c}(x)=E(2cx/\pi)$, and the Fourier coefficient
of $E$ at $k\ne0$ is $\sin(\pi k/3)/(\pi k)$.
Averaging $E$ over the fibers of $x\mapsto gx$ retains only
$\widehat E(gk)$. For $k\ne0$, these vanish if $3\mid g$ and equal
$\chi_6(g)\widehat E(k)/g$ otherwise. Let
\[
P_0=\frac{1+\mathcal A}{12},\qquad P_*=P(u,v,2c,6c).
\]
Then
\[
P(a,b,2c,6c)=
\begin{cases}
P_0,&3\mid g,\\[6pt]
P_0+\dfrac{\chi_6(g)}g(P_*-P_0),&3\nmid g.
\end{cases}
\]
Since $0\le P_*\le(1+\mathcal A)/4=3P_0$,
\[
P\ge
\begin{cases}
P_0\ge\dfrac1{14},&3\mid g,\\[6pt]
\left(1-\dfrac2g\right)P_0\ge\dfrac3{70}>\dfrac1{24},&3\nmid g,
\end{cases}
\]
where $g\ge5$ in the second case.
We may therefore assume
\begin{equation}\label{eq:E4}
a,b,c\ \text{pairwise coprime}.
\end{equation}

\emph{The nonresonant case.} Let $t=x/(2\pi)$ and $N=4c$.
In the normalized variable,
\[
\mathcal C_c=\left\{t=\frac{j+y}N:
 j\in\Z/N\Z,\ |y|<\frac16\right\}.
\]
Thus
\[
\begin{gathered}
\rho=6\int_0^{1/6}G(y)\,dy,\\
G(y)=\frac1N\sum_{j=0}^{N-1}
s\Bigl(\frac{2\pi a(j+y)}N\Bigr)
s\Bigl(\frac{2\pi b(j+y)}N\Bigr),
\end{gathered}
\]
where $G(-y)=G(y)$ follows by replacing $j$ with $-j$.
Let $q(n)=1$ when
$n\bmod N\in I:=\{-c,\ldots,c-1\}$, and $q(n)=-1$ otherwise.
For every integer $n$ and $0<\theta<1$,
$s(2\pi(n+\theta)/N)=q(n)$. Consequently, for
$y\in(0,1/6)$ outside a finite set,
\begin{equation}\label{eq:E7}
G(y)=\frac1N\sum_j
q(aj+\lfloor ay\rfloor)\,q(bj+\lfloor by\rfloor).
\end{equation}
By \eqref{eq:E4}, both $a$ and $b$ are units modulo $N$.
If $a\not\equiv\pm b\pmod{4c}$, necessarily $c\ge2$.
Reindexing \eqref{eq:E7} expresses $G(y)$ as the correlation of
the sign functions associated with $I$ and $ba^{-1}I+d$ for some
$d\in\Z/N\Z$. Lemma~\ref{lem:evenhalf} therefore gives
$G(y)\ge-1/2$, and hence $\rho\ge-1/2$.

The inequality is strict when $b\ne3a$.
For $c\ge5$, the stronger part of Lemma~\ref{lem:evenhalf} gives
$G(y)\ge-1/2+1/(2c)$.
For $c=3$, the same conclusion follows from the fact that the
symmetric difference of two sets of equal size has even cardinality.
For the remaining moduli, let
\[
r\equiv ba^{-1}\pmod{4c},\qquad
R_r(d)=\frac1{4c}\sum_jq(j)q(rj+d).
\]
Then
$G(y)=R_r(\lfloor by\rfloor-r\lfloor ay\rfloor)$.
For $c=4$, direct counting gives
\[
\renewcommand{\arraystretch}{1.3}
\begin{array}{crrrrrr}
\toprule
r&3&5&7&9&11&13\\\midrule
R_r(0)&-1/4&0&-1/4&0&-1/4&0\\
\bottomrule
\end{array}
\]
Thus $G>-1/2$ on an interval immediately to the right of zero.
For $c=2$, the two possibilities are $r=3,5$, with $R_5(0)=0$ and
\[
\renewcommand{\arraystretch}{1.3}
\begin{array}{c*{8}{r}}
\toprule
d&0&1&2&3&4&5&6&7\\\midrule
R_3(d)&-1/2&0&-1/2&0&1/2&0&1/2&0\\
\bottomrule
\end{array}
\]
If $r=3$ and $b\ge7$, then on the nonempty interval
\[
\frac1b<y<\min\left\{\frac1a,\frac2b,\frac16\right\}
\]
we have $G(y)=R_3(1)=0$.
For $b\le5$, the only pair giving $r=3$ is $(a,b)=(1,3)$,
which was excluded. Therefore $\rho>-1/2$ throughout the
nonresonant case under consideration.

\emph{The case $c=1$.} Return to the angular variable.
The set $\mathcal C_1$ consists of four
intervals of radius $\pi/12$ centered at $j\pi/2$.
For odd $a$, we have
$s(\theta+a\pi/2)=-\chi(a)\sgn\sin\theta$. Thus
\[
\frac14\sum_{j=0}^3F(x+j\pi/2)
=\frac12\bigl(s(ax)s(bx)
+\chi(a)\chi(b)\sgn\sin(ax)\sgn\sin(bx)\bigr).
\]
The right-hand side is $K_+$ or $K_-$ with $(\alpha,\beta)=(a,b)$.
By evenness and Lemma~\ref{lem:quadrant},
\[
\rho=\frac{12}\pi\int_0^{\pi/12}K_\pm(x)\,dx\ge-\frac12.
\]
Equality is impossible. Let $L=\operatorname{lcm}(a,b,2,6)$.
All zeros in the normalized variable lie in $(4L)^{-1}\Z$, and
reflection gives $P\in(2L)^{-1}\Z$.
Since $a,b$ are odd, $L$ is divisible by $2$ but not by $4$;
hence $P\ne1/24$.

\emph{The resonant case with $c\ge2$.}
Suppose $a\equiv\pm b\pmod N$. The discrete autocorrelation of $q$ is
\[
\mathcal T_N(n):=\frac1N\sum_jq(j)q(j+n)
=1-\frac4N\dist(n,N\Z).
\]
The same formula extends $\mathcal T_N$ to a $1/c$-Lipschitz function
of a real variable. If $b-a=kN$, reindexing \eqref{eq:E7} gives
\[
G(y)=\mathcal T_N(\lfloor by\rfloor-\lfloor ay\rfloor).
\]
If $a+b=kN$, the identity $q(-n)=q(n-1)$ gives
\[
G(y)=\mathcal T_N(\lfloor ay\rfloor+\lfloor by\rfloor+1).
\]
In both cases $k\ge1$. The absolute differences between the integer
arguments and $kNy$ are, respectively,
$|\{by\}-\{ay\}|$ and $|1-\{ay\}-\{by\}|$.
Each is bounded by
$|\{ay\}-1/2|+|\{by\}-1/2|$.
Since $\mathcal T_N(kNy)=\mathcal T(ky)$, where
$\mathcal T(z)=1-4\dist(z,\Z)$, we obtain
\begin{equation}\label{eq:E19}
\begin{gathered}
\rho\ge\Theta_k-\frac{D_a+D_b}{c},\\
\Theta_k=6\int_0^{1/6}\mathcal T(ky)\,dy,\quad
D_n=6\int_0^{1/6}\left|\{ny\}-\frac12\right|\,dy.
\end{gathered}
\end{equation}
For odd $n$ and $k\ge1$, integration over full periods and the remainder gives
\[
D_n=
\begin{cases}
\dfrac14+\dfrac1{6n},&n\equiv1\pmod6,\\[6pt]
\dfrac14,&n\equiv3\pmod6,\\[6pt]
\dfrac14-\dfrac1{6n},&n\equiv5\pmod6,
\end{cases}
\qquad
\Theta_k=
\begin{cases}
\dfrac2{3k},&k\equiv1,2\pmod6,\\[6pt]
0,&k\equiv0,3\pmod6,\\[6pt]
-\dfrac2{3k},&k\equiv4,5\pmod6.
\end{cases}
\]
Thus $D_1=5/12$ and $D_n\le23/84$ for odd $n\ge3$, while
$\Theta_k\ge-1/6$ and $\Theta_k\ge-2/15$ for $k\ne4$.
Since $a<b$, \eqref{eq:E19} yields
\[
\rho\ge
\begin{cases}
-\dfrac16-\dfrac13\left(\dfrac5{12}+\dfrac{23}{84}\right)
=-\dfrac{25}{63},&c\ge3,\\[6pt]
-\dfrac16-\dfrac12\cdot\dfrac{23}{42}
=-\dfrac{37}{84},&c=2,\ a\ge3,\\[6pt]
-\dfrac2{15}-\dfrac12\cdot\dfrac{29}{42}
=-\dfrac{67}{140},&c=2,\ a=1,\ k\ne4.
\end{cases}
\]
Each bound exceeds $-1/2$.

The remaining case is $c=2$, $a=1$, $k=4$, so $b\in\{31,33\}$.
By the parity shift and reflection,
\[
P(1,b,4,12)=2\bigl|\{t\in[0,1/4]:
\cos(2\pi bt)>0,\ \cos(16\pi t)>1/2\}\bigr|.
\]
The set on the right contains the disjoint intervals
\[
\left(0,\frac1{4b}\right),\qquad
\left(\frac{15}{4b},\frac{17}{4b}\right),
\]
which lie in $(0,1/48)$ and $(5/48,7/48)$, respectively.
Consequently,
\[
P(1,b,4,12)\ge\frac3{2b}\ge\frac1{22}>\frac1{24}.
\qedhere
\]
\end{proof}

\section{Proof of the four-frequency theorems}\label{sec:final}

\begin{proof}[Proof of Theorems~\ref{thm:main} and~\ref{thm:mixed}]
By dilation, assume the quadruple is primitive.
If three frequencies have a common factor greater than $1$,
\eqref{eq:common-factor-bound} gives $P\ge1/27$;
in mixed parity, \eqref{eq:mixed-transfer} and
Corollary~\ref{cor:oneodd} give $P\ge1/18$.
If exactly one or three frequencies are odd,
Corollary~\ref{cor:oneodd} again gives $P\ge1/18$.
We may therefore assume \eqref{eq:tripgcd} and consider the two
remaining parity patterns.

For four odd frequencies, Propositions~\ref{prop:notriple}
and~\ref{thm:onetriple} exclude the cases of no $3$-pair and exactly
one $3$-pair. If two $3$-pairs share a frequency, then
$A\supset\{a,3a,9a\}$; condition \eqref{eq:tripgcd} gives $a=1$,
and Equation~\eqref{eq:triple-exclusion} gives $P>1/33$.
Otherwise the two $3$-pairs are disjoint, and their base frequencies
are coprime by \eqref{eq:tripgcd}.
Proposition~\ref{prop:doubletriple}(a) gives $P\ge1/33$, with equality
precisely for $A=\{1,3,11,33\}$.

For two odd and two even frequencies, Proposition~\ref{prop:2o2e}
gives $P>1/24$ if neither pair has ratio $3$.
Otherwise Propositions~\ref{thm:oddtriple} and~\ref{thm:eventriple}
give $P\ge1/24$, with equality precisely for
\[
A=\{1,3,4,12\}\quad\text{or}\quad A=\{2,4,5,15\}.
\]
Restoring the common divisor proves both theorems. In mixed parity
this divisor is odd, so the parity pattern is preserved.
\end{proof}

\begin{remark}
The equality $P(1,3,11,33)=1/33$ also follows directly from the
common positive set. On $|x|<\pi/6$, where $\cos x$ and $\cos3x$
are positive, the common positive set of $\cos11x$ and $\cos33x$ is
$|x|<\pi/66$.
The common negative set is its translate by $\pi$.
\end{remark}

\appendix

\section{Common positive intervals for odd frequencies}\label{sec:odd-certificates}

We give the interval calculations used in
Lemmas~\ref{lem:triples} and~\ref{lem:odd-candidates}.
Throughout, $t=x/(2\pi)$. For an odd frequency $\tau$,
\begin{equation}\label{eq:positive-arc-certificate}
\cos(2\pi\tau t)>0
\quad\Longleftrightarrow\quad
\frac{4j-1}{4\tau}<t<\frac{4j+1}{4\tau}
\quad\text{for some }j\in\Z.
\end{equation}
Thus every interval below can be checked using rational inequalities
between its endpoints and those in \eqref{eq:positive-arc-certificate}.

\subsection{Intervals for triples}\label{app:odd-triple-certificates}

For an odd triple $T$, let $\mathcal I(T)$ be the component intervals of
\[
\{t\in(0,1/2):\cos(2\pi\tau t)>0\text{ for every }\tau\in T\},
\]
and write
\[
L(T)=\sum_{I\in\mathcal I(T)}|I|,\qquad
k(T)=|\mathcal I(T)|,\qquad M(T)=4k(T)-2.
\]
For each listed triple, the first interval starts at $0$ and none ends
at $1/2$. Reflection about $0$ gives $2k(T)-1$ common positive arcs
on the circle. The common negative set is their translate by $1/2$;
hence
\begin{equation}\label{eq:triple-data}
P(T)=4L(T),\qquad N(T)\le M(T).
\end{equation}
The interval lists in Table~\ref{tab:odd-triples} are exhaustive: order
$0$, $1/2$, and the zeros $(2m+1)/(4\tau)$ lying between them, for
$\tau\in T$, and apply \eqref{eq:positive-arc-certificate} on each
successive interval. The remaining columns follow by adding lengths
and using \eqref{eq:triple-data}.

\begingroup
\setlength{\tabcolsep}{5pt}
\renewcommand{\arraystretch}{1.5}
\begin{longtable}{@{}>{$\displaystyle}l<{$} >{$\displaystyle}l<{$} >{$\displaystyle}c<{$} >{$\displaystyle}c<{$} >{$\displaystyle}c<{$}@{}}
\caption{\small Common positive intervals for the triples in Lemma~\ref{lem:triples}.}
\label{tab:odd-triples}\\
\toprule
T&\text{positive intervals}&P(T)&M(T)&\dfrac{2P(T)}{M(T)+4}\\[6pt]
\midrule
\endfirsthead
\toprule
T&\text{positive intervals}&P(T)&M(T)&\dfrac{2P(T)}{M(T)+4}\\[6pt]
\midrule
\endhead
(1,3,5)&\left(0,\frac1{20}\right)&\frac15&2&\frac1{15}\\[6pt]
(1,3,7)&\left(0,\frac1{28}\right)&\frac17&2&\frac1{21}\\[6pt]
(1,3,9)&\left(0,\frac1{36}\right)&\frac19&2&\frac1{27}\\[6pt]
(1,3,11)&\left(0,\frac1{44}\right),\quad\left(\frac3{44},\frac1{12}\right)&\frac5{33}&6&\frac1{33}\\[6pt]
(1,3,19)&\left(0,\frac1{76}\right),\quad\left(\frac3{76},\frac5{76}\right)&\frac3{19}&6&\frac3{95}\\[6pt]
(3,5,9)&\left(0,\frac1{36}\right),\quad\left(\frac7{20},\frac{13}{36}\right)&\frac7{45}&6&\frac7{225}\\[6pt]
(3,5,11)&\left(0,\frac1{44}\right),\quad\left(\frac7{20},\frac{17}{44}\right)&\frac{13}{55}&6&\frac{13}{275}\\[6pt]
(1,11,23)&\left(0,\frac1{92}\right),\quad\left(\frac7{92},\frac9{92}\right),\quad\left(\frac{15}{92},\frac{17}{92}\right)&\frac5{23}&10&\frac5{161}\\[6pt]
(1,3,21)&\left(0,\frac1{84}\right),\quad\left(\frac1{28},\frac5{84}\right)&\frac17&6&\frac1{35}\\[6pt]
(1,7,21)&\left(0,\frac1{84}\right),\quad\left(\frac{11}{84},\frac{13}{84}\right)&\frac17&6&\frac1{35}\\[6pt]
(3,7,21)&\left(0,\frac1{84}\right),\quad\left(\frac{23}{84},\frac{25}{84}\right)&\frac17&6&\frac1{35}\\[6pt]
(1,3,33)&\left(0,\frac1{132}\right),\quad\left(\frac3{132},\frac5{132}\right),\quad\left(\frac7{132},\frac9{132}\right)&\frac5{33}&10&\frac5{231}\\[6pt]
\bottomrule
\end{longtable}
\endgroup

For example, the common positive interval of frequencies $1$ and $3$
in $(0,1/2)$ is $(0,1/12)$. Intersecting it with the positive arcs for
frequency $11$ gives
\[
\mathcal I(1,3,11)
=\left\{\left(0,\frac1{44}\right),
\left(\frac3{44},\frac1{12}\right)\right\},
\]
so
\[
P(1,3,11)=4\left(\frac1{44}+\frac1{12}-\frac3{44}\right)
=\frac5{33},\qquad M(1,3,11)=6.
\]
All other rows follow by the same intersections.

\subsection{Intervals for quadruples}\label{app:odd-candidate-certificates}

For each $A\in\mathcal E$, Table~\ref{tab:odd-candidates} lists a
union of disjoint common positive intervals. The lists need not be
exhaustive; their total length $L$ suffices to give $P(A)\ge4L$.
Each interval lies within a positive arc of every frequency in $A$,
as follows directly from \eqref{eq:positive-arc-certificate}.

\begin{table}[htbp]
\centering
\caption{\small Common positive subintervals for the ten candidates.}
\label{tab:odd-candidates}
$
\renewcommand{\arraystretch}{1.6}
\begin{array}{@{}>{\displaystyle}l >{\displaystyle}l >{\displaystyle}c@{}}
\toprule
A&\text{positive intervals}&4L\\
\midrule
(1,3,21,35)&\left(\frac1{20},\frac5{84}\right)&\frac4{105}\\[6pt]
(1,3,23,33)&\left(0,\frac1{132}\right),\quad\left(\frac3{92},\frac5{132}\right)&\frac{13}{253}\\[6pt]
(1,3,33,55)&\left(\frac7{220},\frac5{132}\right),\quad\left(\frac7{132},\frac{13}{220}\right)&\frac8{165}\\[6pt]
(1,11,15,33)&\left(0,\frac1{132}\right),\quad\left(\frac{11}{60},\frac{25}{132}\right)&\frac3{55}\\[6pt]
(1,11,19,33)&\left(0,\frac1{132}\right),\quad\left(\frac7{76},\frac{13}{132}\right)&\frac{35}{627}\\[6pt]
(1,11,27,33)&\left(\frac{19}{108},\frac{25}{132}\right)&\frac{16}{297}\\[6pt]
(1,11,31,33)&\left(\frac{11}{124},\frac{13}{132}\right)&\frac{40}{1023}\\[6pt]
(1,11,33,39)&\left(\frac{23}{132},\frac{29}{156}\right)&\frac{20}{429}\\[6pt]
(1,15,19,45)&\left(\frac{15}{76},\frac{37}{180}\right)&\frac{28}{855}\\[6pt]
(3,11,13,33)&\left(0,\frac1{132}\right),\quad\left(\frac{19}{52},\frac{49}{132}\right)&\frac{23}{429}\\[6pt]
\bottomrule
\end{array}
$
\end{table}

The smallest entry in the last column is
\[
\frac{28}{855}>\frac1{33},
\]
which proves the bound used in Lemma~\ref{lem:odd-candidates}.
For instance, for $A=(1,15,19,45)$ the displayed interval
$(15/76,37/180)$ is contained in positive arcs with indices
$j=0,3,4,9$, respectively, in \eqref{eq:positive-arc-certificate}.
Its length is $7/855$, giving $4L=28/855$.

\section{Arithmetic reduction for a single pair in ratio three}\label{sec:odd-arithmetic}

We prove Proposition~\ref{prop:odd-reduction}. The proof uses the
pair-correlation formula, the short-ratio estimates of
Section~\ref{sec:bernoulli}, and Lemma~\ref{lem:triple-elimination}.

\begin{proof}[Proof of Proposition~\ref{prop:odd-reduction}]
The hypothesis is
\begin{equation}\label{eq:low}
P(A)\le\frac1{33}.
\end{equation}
Since $A$ is primitive, \eqref{eq:common-factor-bound} and
\eqref{eq:low} imply \eqref{eq:tripgcd}. Write $A=(a,3a,b,c)$.
Applying \eqref{eq:tripgcd} to $\{a,3a,b\}$, $\{a,3a,c\}$, and
$\{3a,b,c\}$ gives
\begin{equation}\label{eq:A21}
\gcd(a,b)=\gcd(a,c)=1,\qquad 3\nmid\gcd(b,c).
\end{equation}

Thus $b,c$ cannot both be divisible by $3$. In the two parametrizations
below, $d$ is the greatest common divisor of the other two frequencies (after removing
one factor $3$ in Class II), and $D=ad$. The contribution of the four
cross-pairs takes the form
\[
R_D=\frac{2\chi(D)}{3D}
\left(\frac{\chi(u)}u+\frac{\chi(v)}v\right),\qquad
|R_D|\le\frac4{3D\min(u,v)}.
\]
By \eqref{eq:lowconseq}, the correlations other than $C(a,3a)=-1/3$
must contribute at most $-1/11$. We first bound the reduced ratio of
the other two frequencies, then bound $D\min(u,v)$ to obtain a small
triple. When the smaller reduced
frequency is $1$, the bounds in \eqref{eq:triple-fourth} control the
remaining frequency.

\medskip\noindent\emph{Class I: $3\nmid bc$.}
 
Write $b=du$ and $c=dv$, where $u<v$ and $\gcd(u,v)=1$, and let $D=ad$.
Then $3\nmid duv$ and $\gcd(a,duv)=1$.
Since $\gcd(3a,du)=\gcd(3a,dv)=1$, Lemma~\ref{lem:pair} gives
\begin{equation}\label{eq:SI}
S=-\frac13+\frac{\chi(uv)}{uv}+\frac{2\chi(D)}{3D}\Bigl(\frac{\chi(u)}u+\frac{\chi(v)}v\Bigr).
\end{equation}
Here
\[
C(a,du)+C(3a,du)=\frac{2\chi(adu)}{3adu},
\]
and similarly for $dv$. By \eqref{eq:lowconseq}, the last two terms
of \eqref{eq:SI} sum to at most $-1/11$.

\emph{Case I.1: $\chi(u)=\chi(v)$.} The second term is positive, so $\chi(Du)=-1$ and
\begin{equation}\label{eq:A33}
D\le\frac{22(u+v)}{3(uv+11)}<\frac{44}{3u}.
\end{equation}
The strict inequality follows from $u<v$.
If $D=1$, then $a=d=1$, $\chi(u)=-1$, $3\nmid u$, and $u\le14$.
Thus $u\in\{7,11\}$.
The value $u=7$ gives $A\supset\{1,3,7\}$, which is excluded by
\eqref{eq:triple-exclusion}.
For $u=11$, \eqref{eq:A33} gives $v\le19$.
The conditions $v>11$, $v\equiv3\pmod4$, and $3\nmid v$ then force $v=19$,
contrary to $\max A\ge33$ in \eqref{eq:lowconseq}.

If $D\ge3$, then $u<44/9$, so $u=1$.
Hence $\chi(D)=-1$ and $D<22/3$, giving $D\in\{3,7\}$.
For $D=3$, the condition $3\nmid d$ forces $(a,d)=(3,1)$, and
$A\supset\{1,3,9\}$.
For $D=7$, the choice $(a,d)=(1,7)$ gives $A\supset\{1,3,7\}$.
Both are excluded by \eqref{eq:triple-exclusion}.
The other choice, $(a,d)=(7,1)$, gives
$A=(7,21,1,v)\supset\{1,7,21\}$.
But \eqref{eq:A33} requires $v\ge209$, contrary to
\eqref{eq:triple-fourth}.

\emph{Case I.2: $\chi(v)=-\chi(u)$.} Here $uv\equiv3\pmod4$ and
$3\nmid uv$. If $uv\in\{7,11\}$, then $(b,c)=(d,dv)$ with
$v\in\{7,11\}$, and \eqref{eq:A21} with $3\nmid d$ gives the
hypotheses of Proposition~\ref{prop:shortodd}. Hence $P>1/33$.
For $uv\ge19$, \eqref{eq:lowconseq} forces $\chi(Du)=-1$ and
\begin{equation}\label{eq:A34}
D\le\frac{22(v-u)}{3(uv-11)} .
\end{equation}
If $u\ge5$, the right side is $<22/(3u)$, since $u^2>11$.
Thus $D=1$ and $u\in\{5,7\}$.
The sign condition $\chi(u)=-1$ forces $u=7$, so
$A\supset\{1,3,7\}$, contrary to \eqref{eq:triple-exclusion}.

We are left with $u=1$.
Then $\chi(D)=-1$, $v\equiv3\pmod4$, $3\nmid v$, and $v\ge19$.
The right side of \eqref{eq:A34} decreases with $v$, so
\[
D\le\frac{22\cdot18}{3\cdot8}=\frac{33}2,
\qquad D\in\{3,7,11,15\}.
\]
Since $3\nmid d$, the corresponding factorizations are
\[
\begin{array}{c|c}
D&(a,d)\\
\hline
3&(3,1)\\
7&(1,7),\ (7,1)\\
11&(1,11),\ (11,1)\\
15&(3,5),\ (15,1)
\end{array}
\]
For $D=3$, the triple $\{1,3,9\}\subset A$ is excluded by
\eqref{eq:triple-exclusion}. For $D=7$, the choice $(a,d)=(1,7)$
contains $\{1,3,7\}$, while $(a,d)=(7,1)$ gives
\[
A=(1,7,21,v),\qquad
33\le v\le36,\qquad v\equiv3\pmod4.
\]
Thus $v=35$, contrary to $\gcd(a,v)=1$.

For $D=11$, the choice $(a,d)=(1,11)$ gives
$\{1,3,11\}\subset A$ with fourth frequency $11v>33$, contradicting
\eqref{eq:triple-fourth}. The other choice gives, by \eqref{eq:A34},
\[
A=(1,11,33,v),\qquad v\le31,\qquad v\in\{19,23,31\}.
\]
The value $v=23$ is excluded by \eqref{eq:triple-exclusion}.
For $D=15$, \eqref{eq:A34} gives
\[
v\le\frac{473}{23},\qquad v=19.
\]
The choice $(a,d)=(3,5)$ contains $\{3,5,9\}$ and is excluded by
\eqref{eq:triple-exclusion}; $(a,d)=(15,1)$ gives $(1,15,19,45)$.
Thus Class I gives $A\in\mathcal E_{\mathrm I}$, as defined in
\eqref{eq:A35}.

\medskip\noindent\emph{Class II: exactly one of $b,c$ is divisible by $3$.}

Label the frequencies so that $3\mid b$ and $3\nmid c$.
Write $b=3du$ and $c=dv$, where $\gcd(u,v)=1$ and $3\nmid v$, and let $D=ad$.
Since $\gcd(a,b)=1$ and $3\nmid c$, we have $3\nmid D$;
also $\gcd(a,duv)=1$.
Since
\[
\gcd(3du,dv)=d,\quad \gcd(a,3du)=1,\quad
\gcd(3a,3du)=3,\quad \gcd(3a,dv)=1,
\]
Lemma~\ref{lem:pair} gives
\begin{equation}\label{eq:SII}
S=-\frac13-\frac{\chi(uv)}{3uv}+\frac{2\chi(D)}{3D}\Bigl(\frac{\chi(u)}u+\frac{\chi(v)}v\Bigr).
\end{equation}
Let $r=\min(u,v)$ and $s=\max(u,v)$. If $r=s$, then $u=v=1$ and
$(c,b)=(d,3d)$ is a second $3$-pair. Thus $r<s$.

\emph{Case II.1: $\chi(r)=\chi(s)$.} Then $rs\ge5$, and \eqref{eq:lowconseq} forces
\begin{equation}\label{eq:A43}
\chi(Dr)=-1,\qquad D\le\frac{22(r+s)}{3rs-11}.
\end{equation}
The right side decreases in each of $r$ and $s$.

\emph{Subcase $r\ge3$.}
Then $s\ge7$ and $D\le220/52<5$.
Since $D$ is odd and $3\nmid D$, we obtain $D=1$; hence $a=d=1$ and
$\chi(r)=-1$.
The expression $3rs-22(r+s)-11$ is increasing in each variable for
$r,s\ge15$ and equals $4$ at $r=s=15$. This excludes $r\ge15$.
Thus $r\in\{3,7,11\}$.
For $r=3$, the condition $3\nmid v$ forces $u=3$, so
$\{1,3,9\}\subset A$, contrary to \eqref{eq:triple-exclusion}.
For $r=7$, the choice $v=7$ contains $\{1,3,7\}$ and is similarly
excluded. If $u=7$, then
\[
A=(1,3,21,s),\qquad 33\le s\le36,\qquad s\equiv3\pmod4,
\]
by \eqref{eq:triple-fourth}; hence $s=35$, contrary to $\gcd(7,s)=1$.

For $r=11$, the choice $v=11$ contains $\{1,3,11\}$, so
\eqref{eq:triple-fourth} would force $3s=33$, contrary to $s>r$.
If $u=11$, \eqref{eq:A43} gives
\[
A=(1,3,33,s),\qquad 11<s\le23,\qquad
s\equiv3\pmod4,\qquad 3\nmid s.
\]
Thus $s\in\{19,23\}$. The value $s=19$ is excluded by
\eqref{eq:triple-exclusion}, leaving $(1,3,23,33)$.

\emph{Subcase $r=1$.}
Here $s\equiv1\pmod4$, $s\ge5$, $\chi(D)=-1$, and $3\nmid D$.
The bound in \eqref{eq:A43} decreases with $s$ and equals $33$ at $s=5$.
Thus
\[
D\in\{7,11,19,23,31\},\qquad (a,d)\in\{(1,D),(D,1)\}.
\]
For $D=7$, \eqref{eq:A43} gives no further restriction;
for $D=11$, it gives $s\in\{5,9,13\}$;
for $D\in\{19,23,31\}$, it gives $s=5$.

For $D=7$, the four arrangements reduce as follows. Here the bounds
on the fourth frequency follow from \eqref{eq:triple-fourth}.
\[
\renewcommand{\arraystretch}{1.25}
\begin{array}{ccl}
(a,d)&(u,v)&\text{consequence}\\
\hline
(1,7)&(1,s)&A=(1,3,21,7s),\quad s=5,\\
(1,7)&(s,1)&\{1,3,7\}\subset A,\\
(7,1)&(1,s)&A=(3,7,21,s),\quad s=33,\\
(7,1)&(s,1)&A=(1,7,21,3s),\quad33\le3s\le36.
\end{array}
\]
The second row is excluded by \eqref{eq:triple-exclusion}; the last
two contradict $3\nmid v$ and $s\equiv1\pmod4$, respectively. Thus
only $(1,3,21,35)$ remains.

For $D=11$, the same argument gives
\[
\renewcommand{\arraystretch}{1.25}
\begin{array}{ccl}
(a,d)&(u,v)&\text{consequence}\\
\hline
(1,11)&(1,s)&A=(1,3,33,11s),\quad11s\le55,\\
(1,11)&(s,1)&A=(1,3,11,33s),\quad33s=33,\\
(11,1)&(1,s)&A=(3,11,33,s),\quad s\in\{5,13\},\\
(11,1)&(s,1)&A=(1,11,33,3s),\quad3s\in\{15,27,39\}.
\end{array}
\]
The second row contradicts $s\ge5$; in the third row, $s=5$ is
excluded by \eqref{eq:triple-exclusion}. The surviving quadruples are
\[
(1,3,33,55),\quad(3,11,13,33),\quad
(1,11,15,33),\quad(1,11,27,33),\quad(1,11,33,39).
\]
For $D\in\{19,23,31\}$ and $s=5$ the four quadruples are
\[
(1,3,3D,5D),\quad(1,3,D,15D),\quad(3,5,D,3D),\quad(1,15,D,3D).
\]
For $D\ge19$, $D\equiv3\pmod4$, and $\gcd(D,15)=1$,
Lemmas~\ref{lem:expansion} and~\ref{lem:pair} give
\begin{equation}\label{eq:fourtriples}
\begin{aligned}
&P(1,3,3D)=P(1,3,D)=\frac{D-1}{6D},\\
&P(3,5,D)=\frac{7D+1}{30D},\\
&P(1,15,D)=\frac{7(D-1)}{30D}.
\end{aligned}
\end{equation}
The component counts needed in Lemma~\ref{lem:component} are
\[
\begin{array}{@{}ccccc@{}}
\toprule
T&(1,3,3D)&(1,3,D)&(3,5,D)&(1,15,D)\\
\midrule
N(T)\le&D-1&2(D+4)&D-1&D-1\\
\bottomrule
\end{array}
\]
Indeed, the common positive set for $(1,3)$ is $|x|<\pi/6$.
It contains exactly $(D-1)/2$ full positive intervals of $\cos(3Dx)$,
which proves the first count; the second is the general component bound.
For $(3,5)$ the common positive set consists, in the variable $t=x/(2\pi)$,
of the three arcs $|t|<1/20$ and $\pm(7/20,5/12)$, of total length $7/30$.
For a positive interval of $\cos(2\pi Dt)$ to meet an arc of length
$\ell$, its center $j/D$ must lie in the arc enlarged by $1/(4D)$ at
each end. There are at most $D\ell+3/2$ such centers. Thus
\[
N(3,5,D)\le\frac{7D}{15}+9\le D-1.
\]
Finally, $|x|<\pi/2$ contains $(D-1)/2$ full positive intervals of
$\cos(Dx)$. Each has length $\pi/D<\pi/15$, so it meets the positive
set of $\cos(15x)$ in at most one interval. This proves the last count.
The corresponding negative sets are translates by $\pi$.
Substitution into Lemma~\ref{lem:component} gives, in the order displayed,
\[
P(A)\ge\frac{D-1}{30D},\quad
\frac{3D-13}{60D},\quad
\frac{D+3}{30D},\quad
\frac{D-1}{30D}.
\]
All four quantities exceed $1/33$ for $D\ge19$.

\emph{Case II.2: $\chi(s)=-\chi(r)$.} Now $\chi(uv)=-1$ and \eqref{eq:lowconseq} forces
\begin{equation}\label{eq:A48}
\chi(Dr)=-1,\qquad D\le\frac{22(s-r)}{3rs+11}<\frac{22}{3r}.
\end{equation}
If $r\ge3$, then $D=1$, $\chi(r)=-1$, and $r<22/3$.
Thus $r\in\{3,7\}$.
For $r=3$, the condition $3\nmid v$ gives $u=3$ and
$A\supset\{1,3,9\}$.
For $r=7$, the choice $v=7$ gives $A\supset\{1,3,7\}$.
These are excluded by \eqref{eq:triple-exclusion}.
The choice $u=7$ gives $A=(1,3,21,s)$, but \eqref{eq:A48} requires
$s\ge165$, contrary to \eqref{eq:triple-fourth}.

If $r=1$, then $D<22/3$, $\chi(D)=-1$, and $3\nmid D$, so $D=7$.
Equation \eqref{eq:A48} now gives $s\ge99$.
The four arrangements are $(1,3,21,7s)$, $(1,3,7,21s)$,
$(3,7,21,s)$, and $(1,7,21,3s)$.
The second contains $(1,3,7)$; each of the others contains a triple from
the second line of \eqref{eq:triple-fourth} with fourth frequency $\ge99$.
None satisfies \eqref{eq:low}.

Class II therefore gives $A\in\mathcal E_{\mathrm {II}}$, as defined in
\eqref{eq:A49}. Together with Class I, this proves
$A\in\mathcal E_{\mathrm I}\cup\mathcal E_{\mathrm {II}}=\mathcal E$.
\end{proof}

\medskip
\Needspace{6\baselineskip}
\noindent\textbf{Disclosure on artificial intelligence (AI) assistance:}
The authors acknowledge the use of OpenAI tools in the preparation of
this manuscript, including language editing, \LaTeX{} formatting,
organization, and assistance with mathematical derivations and proof
checks. The authors retain full responsibility for the final verification
of all mathematical statements and proofs.

\end{document}